\documentclass[12pt,letterpaper]{article}

\usepackage{amsmath, amssymb, amsthm}
\usepackage{mathrsfs}
\usepackage{mathtools}
\usepackage{subcaption}
\usepackage{cancel}
\usepackage{xcolor}
\usepackage{bbm} 
\usepackage{hyperref}
\usepackage{bm}
\hypersetup{colorlinks = true, linkcolor=red, citecolor=blue}

\usepackage{kyle_page_setup}
\usepackage{kyle_math_general_definitions}

\newcommand{\cB}{\mathcal{B}}
\newcommand{\cH}{\mathcal{H}}

\newcommand{\ii}{\mathrm{i}}

\newcommand{\N}{\mathbb{N}}
\newcommand{\R}{\mathbb{R}}
\newcommand{\UU}{\mathbb{U}}
\newcommand{\Z}{\mathbb{Z}}

\newcommand{\create}[1]{{a}_{#1}^\dagger}
\newcommand{\annihi}[1]{{a}_{#1}}

\usepackage[numbers]{natbib}

\definecolor{kylegreen}{rgb}{0,.7,0}
\definecolor{misson}{rgb}{1,.2,.2}

\usepackage{graphicx}
\usepackage{epstopdf}
\usepackage{epsfig}

\begin{document}

\title{Dissociation limit of the H$_2$ molecule in phRPA-DFT}
\author{Mi-Song Dupuy and Kyle Thicke}
\maketitle

\abstract{Despite the simplicity of the H$_2$ system, the correct dissociation of H$_2$ is known to be a difficult problem for density functional approximations. In this work, we consider the particle-hole random phase approximation (phRPA), an approximation to the correlation energy in electronic structure, and show that the phRPA energy of the H$_2$ molecule correctly dissociates.  
That is, as the hydrogen atoms are pulled apart, the phRPA energy of the system converges to twice the phRPA energy of a single hydrogen atom.  
As part of our result, we prove that the phRPA correlation energy 
is well-defined.}

\section{Introduction}

The electronic properties of a molecule is encoded in the lowest eigenfunction of the $N$-body Schr\"odinger operator. 
Except for the smallest molecules, due to the high-dimensional nature of the equation, solving directly the eigenvalue problem is a colossal numerical challenge.
Admiringly, the lowest eigenvalue of the $N$-body Schr\"odinger operator can be obtained by a minimization of a functional depending only on the electronic density, i.e., a function of a single space variable. This is the \emph{density functional theory} discovered by Hohenberg and Kohn \cite{hohenberg1964inhomogeneous}.
This tremendous reduction of dimensionality comes with a catch: this functional is unknown and only approximations in specific asymptotic regimes are available.
Nevertheless, satisfactory approximations have been designed, partially relying on these asymptotic behaviors, which can tackle molecules with thousands of electrons. 

In this whole variety of density functional approximations, only a few are able to describe the dissociation of the simplest molecule H$_2$. 
Physically, by stretching the H$_2$ molecule, we expect the energy of the whole system in the limit to be the sum of the energy of the single H atoms. 
So far this has only been achieved by the Strictly Correlated Electron (SCE) density functional \cite{chen2014numerical} and random phase approximation (RPA) \cite{cohen2012challenges}.
The latter model and in particular the particle-hole RPA (phRPA) is the main focus of the present paper, where we show that the phRPA correlation energy is well-defined and gives the correct dissociation limit.

The particle-hole random phase approximation (phRPA) is typically known simply as RPA in the physics and chemistry literature.  It originated in an investigation of the response of the homogeneous electron gas to a wave \cite{bohm1951collective}.  The term ``random phase approximation'' referred to the fact that if there were a lot of electrons that were in random locations, there would be an overall cancellation of the out-of-phase responses.  However, this interpretation is rarely utilized today, as the diagrammatic interpretation (discussed below) has become dominant.  The ``particle-hole'' part of phRPA refers to the fact that the method gives the response function to a perturbation in the density, which can be written as a ``particle'' and a ``hole'' term in second quantization (i.e., $\rho_{ij} = \braopket{\Psi}{\create{i}\annihi{j}}{\Psi}$, where $\create{i}$ and $\annihi{j}$ are the usual creation and annihilation operators).  This distinguishes it from, e.g., the particle-particle RPA \cite{van2013exchange}, which gives the reponse function to a time-dependent perturbation in the so-called pairing field.
In second quantization, the pairing field couples $N$-particle states and $N+2$-particle states giving information on linear response properties involving two particles.

In the physics literature, phRPA is typically derived using Feynman diagrams.  Essentially, Feynman diagrams allow one to keep track of all the terms in pertubation expansion of the Green's function in terms of the Coulomb interaction (each term in the pertubation expansion corresponds to a diagram).  To derive phRPA, physicists consider the uniform electron gas at high density.  A non-rigorous analysis of the diagrams (i.e., the terms in the pertubation expansion), shows that in the high density limit, only the so-call ring diagrams are important.  Therefore, phRPA is taken to be the sum of these ring diagrams (i.e., phRPA is the sum of the terms in the perturbation expansion that correspond to the Feynman diagrams consisting only of rings) \cite{FetterWalecka}.  In this paper, we explore the facinating fact that even though phRPA is derived for the high density limit, it correctly dissociates H$_2$.  For a more in-depth review of phRPA, see \cite{ren2012random}.  



%

\subsection{Short review on dissociation in DFT}

For finite systems and nonrelativistic electrons, the system is characterized by a Hamiltonian $H_N(v_\mathrm{ext},w)$
\begin{equation}
  \label{eq:N-body-Hamiltonian}
  H_N(v_\mathrm{ext},w) =  \sum\limits_{i=1}^N \Big(-\frac{1}{2}\Delta_{r_i} + v_{\mathrm{ext}}(r_i) \Big) + \sum\limits_{1 \leq i < j \leq N} w(r_i-r_j).
\end{equation}
The potentials $v_\mathrm{ext}$ and $w$ are such that $v_\mathrm{ext} \leq 0 \in L^2(\RR^3) + L_\varepsilon^\infty(\RR^3)$ and $w(r) = \tfrac{1}{|r|}$, where the space $L^2(\RR^3) + L_\varepsilon^\infty(\RR^3)$ is the space of functions $v$ such that for all $\varepsilon > 0$, there exist $v_2 \in L^2(\RR^3)$ and $v_\infty \in L^\infty(\RR^3)$ with $v = v_2 + v_\infty$, $\mathrm{supp} \, v_2 \subset B_R$ and $\|v_\infty\|_{L^\infty} \leq \varepsilon$. 

The operator $H_N(v_\mathrm{ext},w)$ as an operator acting on $\bigwedge_{i=1}^N L^2(\RR^3 \times \ZZ_2)$ with domain $\bigotimes_{i=1}^N H^2(\RR^3 \times \ZZ_2)$ is self-adjoint \cite{reed_simon_iv}.
For the external potentials considered $v_\mathrm{ext}$, we are going to assume that the ground-state energy $E_0^N$ is nondegenerate.

\begin{assumption}[Uniqueness of the ground-state]
  \label{assumption:unique_GS}
  The ground-state energy $E_0^N$ is a simple eigenvalue of $H_N(v_\mathrm{ext},w)$.
\end{assumption}

The ground state $\Psi_0 \in \bigwedge_{i=1}^N L^2(\RR^3 \times \ZZ_2)$ associated to the lowest eigenvalue $E_0$ of $H_N(v_\mathrm{ext},w)$ is then unique up to a phase factor. 
An important reduced quantity of the ground state wave function is the \emph{electronic density} $\rho$ given by 
\begin{equation}
  \label{eq:electronic_density_definition}
  \rho(r) = \sum_{s_1,\dots,s_N \in \mathbb{Z}_2} \int_{\RR^{3(N-1)}} |\Psi_0(rs_1;\dots;r_Ns_N)|^2 \, \mathrm{d}r_2 \dots \mathrm{d}r_N.
\end{equation}
The ground state $\Psi_0$ and its electronic density has some known properties that will be used subsequently and are listed below. 

\begin{proposition}[Properties of the ground-state of $H_N(v_\mathrm{ext},w)$ \cite{yserentant2010regularity}]
  \label{prop:N_body_Hamiltonian}
  Assume that Assumption~\ref{assumption:unique_GS} holds. Then the following assertions are true:
  \begin{enumerate}
    \item $\Psi_0$ can be chosen real-valued and $\Psi_0 \in \bigotimes_{i=1}^N H^2(\RR^{3} \times \Z_2)$;
    \item $\rho \in L^1(\RR^3) \cap L^\infty(\RR^3)$ and there are constants $a,C>0$ such that for all $r \in \RR^3$, $\rho(r) \leq C e^{a|r|}$.
  \end{enumerate}
\end{proposition}

Due to the computational intractibility of directly solving the Schrödinger equation 
\[
  H_N(v_\mathrm{ext},w) \Psi_0 = E_0 \Psi_0,  
\]
with standard numerical techniques, density functional theory (DFT) has become a very popular method for finding the ground state $\Psi_0$ of a quantum system.  The great advantage of DFT is the Hohenberg-Kohn (HK) theorem \cite{hohenberg1964inhomogeneous,levy1979universal,lieb1983}, which states that external potential of a system of electrons is uniquely determined by the ground state density of the system.  Since the normal formulation of quantum mechanics shows that the external potential determines the ground state wave function $\Psi_0$, which determines all observable ground state properties of the system, the HK theorem implies that we only need the ground state density $\rho$
instead of the $N$-particle ground state wavefunction.  This is a radical reduction in the size of the problem -- from $N$ spatial dimensions to three.

In particular, we are interested in computing the ground state energy of the system.  
In this case, the HK theorem states that there is a functional $F_\mathrm{HK}$ such that the ground state energy $E_0$ can be obtained by a minimization involving the functional $F_\mathrm{HK}$:
\begin{equation}
  E_0 = \inf_{\rho \in \mathcal{I}_N} \big( F_\mathrm{HK}(\rho) + \langle v_\mathrm{ext}, \rho \rangle \big),
\end{equation}
where $\mathcal{I}_N$ is the set of the $N$-representable densities given by 
\[
  \mathcal{I}_N = \Big\lbrace \rho \in L^1(\RR^3) \ \big| \ \rho \geq 0, \int_{\mathbb{R}^3} \rho = N, \sqrt{\rho} \in H^1(\mathbb{R}^3) \Big\rbrace.
\]
$F_\mathrm{HK}$ is defined by the constrained search of possible $L^2$-normalized wave functions $\Psi$ which has a density $\rho$
\begin{equation}
  F_\mathrm{HK}[\rho] = \min_{\Psi \mapsto \rho} \langle \Psi, (-\tfrac{1}{2} \Delta + \sum_{1\leq i,j \leq N}w(r_i-r_j)) \Psi \rangle.
\end{equation}
No explicit closed formula for $F_\mathrm{HK}[\rho]$ exists, so it needs to be approximated.  A common method of approximation is the Kohn-Sham method \cite{kohn-sham}.

In Kohn-Sham DFT, one considers a non-interacting set of electrons which has the same ground state density of the interacting system. The following splitting of the HK functional is then chosen
\begin{equation}
  \label{eq:F_HK_definition}
  F_\mathrm{HK}[\rho] = T_s[\rho] + E_{\mathrm{Hxc}}[\rho],
\end{equation}
where $T_s$ is the Kohn-Sham kinetic energy
\begin{equation}
  \label{eq:Ts_definition}
  T_s[\rho] = \inf_{\Phi \in \mathcal{S}_{N,\rho}} \langle \Phi, -\tfrac{1}{2} \Delta \Phi \rangle,
\end{equation}
and $S_{N,\rho} = \lbrace \Phi = \phi_1 \wedge \dots \wedge \phi_N \ | \ \phi_i \in L^2(\mathbb{R}^3 \times \mathbb{Z}_2), \langle \phi_i, \phi_j \rangle = \delta_{ij}, \sum_{s \in \mathbb{Z}_2} \sum_{i=1}^N |\phi_i|^2(r,s) = \rho(r) \rbrace$.
The functional $E_\mathrm{Hxc}$ accounts for the difference of the HK functional and the Kohn-Sham kinetic energy.

It is common to split $E_\mathrm{Hxc}$ into a Hartree, an exchange and a correlation energy
\begin{equation}
  \label{eq:E_Hxc_definition}
  E_\mathrm{Hxc}[\rho] = E_\mathrm{H}[\rho] + E_x[\rho] + E_c[\rho],
\end{equation}
where 
\begin{align}
  E_H[\rho] + E_x[\rho] &= \Big\langle \Phi,  \sum_{1\leq i<j\leq N} w(r_i-r_j) \Phi \Big\rangle, \\
  E_\mathrm{H}[\rho] &= \frac{1}{2} \int \rho(r)\rho(r') w(r,r') \dd r \dd r', \\
  E_x[\rho] &= -\frac{1}{2} \int \Big| \sum_{s \in \mathbb{Z}_2} \sum_{i=1}^N \phi_i(r,s)\phi_i(r',s) \Big|^2 w(r,r') \dd r \dd r',
\end{align}
where $\Phi_\mathrm{KS} = \phi_1 \wedge \dots \wedge \phi_N \in S_{N,\rho}$ is the minimizer of $T_s[\rho]$.
This particular choice of $E_x$ is usually called the \emph{exact exchange}.  Other options are available, for instance the \emph{local density approximation} (LDA) \cite{dirac1930note} where the exact exchange is approximated for the free electron gas in the thermodynamic limit. In this case, $E_x$ can be written explicitely in terms of the density $\rho$.

The correlation energy $E_c[\rho]$ contains all the error between the Kohn-Sham wave function $\Phi_\mathrm{KS}$ and the HK wave function $\Psi_\mathrm{HK}$
\begin{align}
  E_c[\rho] &= F_\mathrm{HK}[\rho] - \langle \Phi_\mathrm{KS}, (-\tfrac{1}{2} \Delta + \sum_{1\leq i,j \leq N}w(r_i-r_j)) \Phi_\mathrm{KS} \rangle \\
  &= \langle \Psi_\mathrm{HK}, (-\tfrac{1}{2} \Delta + \sum_{1\leq i,j \leq N}w(r_i-r_j)) \Psi_\mathrm{HK} \rangle - \langle \Phi_\mathrm{KS}, (-\tfrac{1}{2} \Delta + \sum_{1\leq i,j \leq N}w(r_i-r_j)) \Phi_\mathrm{KS} \rangle.
\end{align}
The latter, sometimes together with $E_x[\rho]$ has been extensively scrutinized in the search of the best approximation of $F_\mathrm{HK}$ \cite{cohen2012challenges,toulouse2021review}. In practice, $E_c[\rho]$ is computed using the Kohn-Sham orbitals $(\phi_i)_{i=1,\dots,N}$
\begin{equation}
  E_c[\rho] \approx E_c[(\phi_i)].
\end{equation}
%

Although the dissociation problem is a classical question in quantum chemistry, mathematical works on the dissociation limit is scarce.  

The dissociation limit for LDA in DFT has been studied in \cite{behr2020dissociation} in a spinless case. 
In dimension one, for a contact potential, it is shown that symmetric splitting occurs under the assumption that the exchange constant is not too large. 
In dimension three, the correct symmetric splitting is proved only for positively charged systems.

In \cite{holst2021symmetry}, it is established that for the H$_2$ molecule in the spin-LDA model, that minimizers at a fixed distance $R$ can break the spin symmetry, depending again on the strengh of the exchange constant. 
This is related to the dissociation limit problem as a wrong limit is due to a spurious self interaction, which increases the total energy but can be solved by considering spin-polarized models.

Finally, let us mention that the SCE functional \cite{seidl2007strictly,gori2009density} which takes as a reference system the complete correlation of the electrons, correctly dissociates H$_2$ as shown in \cite{chen2014numerical}.

\subsection{Structure of the paper}

In this work, we prove that DFT does correctly dissociate (see Theorem~\ref{thm:phrpa_dissociation}) when we approximate $E_{c}$ with the phRPA correlation energy, \emph{i.e.} we show that the energy of the two-hydrogen system approaches that of two isolated H atoms as the atoms are pulled apart,
\begin{equation}
	\lim_{R\to\infty} \mathcal{E}^\text{phRPA}(\text{H}_2) = 2\,\mathcal{E}^\text{phRPA}(\text{H}),
\end{equation}
where $R$ is the distance between the atoms in the two-atom system and $\mathcal{E}^\text{phRPA}$ is the total phRPA energy defined in Equation~\eqref{eq:def_total_phrpa_energy}, using the phRPA approximation for the correlation energy.

In Section~\ref{sec:def_phrpa_corr_energy}, we derive the expression for the phRPA correlation energy and show that the resulting expression is well defined under the assumption of a spectral gap for the one-body model. 
In Section~\ref{sec:dissociation_Nbody_rhf}, we show that the dissociation of H$_2$ is exact in the $N$-body Schr\"odinger model as opposed to the restricted Hartree-Fock case. 
Finally, in Section~\ref{sec:phrpa_dissociation}, we prove that the phRPA correlation energy with the restricted Hartree-Fock orbitals correctly dissociates H$_2$.

\section{Derivation of the phRPA correlation energy}
\label{sec:def_phrpa_corr_energy}

In this section, we are going to derive the phRPA correlation energy for a general $N$-body molecular system and precisely state the assumptions needed and the approximations made to derive the formula of the correlation energy. 
Since it involves the retarded linear response function $\chi$, we introduce this operator.

\subsection{The retarded linear response operator $\chi$}

\subsubsection{Definition of $\chi$}

The retarded linear response operator $\chi$ is the linear response of an interacting system at equilibrium to a one-body perturbation. 
To be more specific, let us consider $\varepsilon>0$, $f \in L^\infty(\RR)$ such that $f(t) = 0$ for $t<0$ and $\beta \in C^{\infty}_c(\RR^3)$. 
Let $\Psi \in \bigwedge_{i=1}^N L^2(\RR^3 \times \ZZ_2)$ be the solution to the time-dependent Schrödinger equation with a perturbation that is turned on at time $t=0$, 
\begin{equation}
  \label{eq:TD_schrodinger}
  \left\lbrace
  \begin{aligned}
    \ii \frac{\partial \Psi}{\partial t}(x,t) & = H_N(v_\mathrm{ext},w) \Psi(x,t) + \varepsilon f(t) \sum\limits_{i=1}^N \beta(r_i) \Psi(x,t), \quad \forall x \in \big( \RR^3 \times \ZZ_2 \big)^N, t>0 \\
    \Psi(x,0) &= \Psi_0(x),
  \end{aligned}
  \right.
\end{equation} 
where $H_N(v_\mathrm{ext},w)$ is the $N$-body operator defined in Equation~\eqref{eq:N-body-Hamiltonian}.

Let $\alpha \in C^{\infty}_c(\RR^3)$ be some multiplication operator in space. The retarded linear response function $\chi$ is defined by the linear variations with respect to $\varepsilon$ of the expected value $\langle \Psi(t), \sum_{i=1}^N \alpha(r_i) \Psi(t) \rangle$:
\begin{equation}
  \Big\langle \Psi(t), \sum_{i=1}^N \alpha(r_i) \Psi(t) \Big\rangle = \Big\langle \Psi_0, \sum_{i=1}^N \alpha(r_i) \Psi_0 \Big\rangle + \varepsilon (f \star \langle \alpha, \chi \beta \rangle)(t) + \cO(\varepsilon^2),
\end{equation}
where $\star$ denotes the convolution on $\RR$.

\begin{definition}[Retarded linear response]
  Let $B : L^{6}(\RR^3) \to \bigwedge\limits_{i=1}^N L^2(\RR^3 \times \ZZ^2)$ be given by 
  \begin{equation}
    \label{eq:B_operator}
    Bv(x_1,\dots,x_N) = \Big( \sum\limits_{i=1}^N v(r_i) - \braket{\rho}{v}\Big) \Psi_0(x_1,\dots,x_N),
  \end{equation} 
  where $\Psi_0$ is the ground-state of $H_N(v_\mathrm{ext},w)$ defined in Equation~\eqref{eq:N-body-Hamiltonian}.

  For $\tau \in \RR$, the retarded linear response operator $\chi(\tau) : L^{6}(\RR^3) \to L^{6/5}(\RR^3)$ is defined by 
  \begin{equation}
    \label{eq:chi_rigorous}
    \chi(\tau) = 2 \, \mathrm{Re}\Big(-\ii \theta(\tau) B^* e^{-\ii \tau (H_N(v_\mathrm{ext},w)-E_0)} B \Big),
  \end{equation}
  where $B^* : \bigwedge\limits_{i=1}^N L^2(\RR^3\times \Z_2) \to L^{6/5}(\RR^3)$ is the adjoint of $B$ with $L^2$ as the pivoting space. 
\end{definition}

The retarded linear response operator $\chi$ can rigorously be established for $\alpha, \beta \in C^\infty_c(\RR^3)$ as shown in the appendix (see Proposition~\ref{prop:rigorous_chi_stronger_assumptions}).
Since Equation~\eqref{eq:chi_rigorous} is valid for a larger set of $\alpha$ and $\beta$, we will take it as granted.

\begin{proposition}
  The operator $\chi$ defined in Equation~\eqref{eq:chi_rigorous} is in $L^\infty\big(\RR, \mathcal{B}(L^6(\RR^3),L^{6/5}(\RR^3))\big)$ where $\mathcal{B}(L^6(\RR^3),L^{6/5}(\RR^3))\big)$ is the space of bounded operators from $L^6(\RR^3)$ to $L^{6/5}(\RR^3)$.
\end{proposition}

\begin{proof}
  Since $H_N(v_\mathrm{ext},w)$ is self-adjoint, it is enough to show that $B$ is a bounded operator from $L^{6}(\RR^3)$ to $L^2(\RR^3)$.
  Let $v \in L^6(\RR^3)$ then we have
  \begin{equation}
    \| Bv \|_{L^2} \leq C \| v \|_{L^6} \|\rho\|_{L^{3/2}},
  \end{equation}
  for a constant $C$ independent of $v$ or $\rho$.
  The result follows by Proposition~\ref{prop:N_body_Hamiltonian}
\end{proof}

Other choices of spaces are possible. Our choice is motivated by the role played later by the operator $(v_c^{1/2})^* \chi v_c^{1/2}$ in the definition of the phRPA correlation energy, where $v_c$ is the square root of the Hartree operator $v_c$ defined by $v_cf(r) = \int_{\RR^3} \frac{f(r')}{|r-r'|} \, \mathrm{d}r'$. 
In the Fourier representation, $v_c$ is the multiplication operator by $\tfrac{4\pi}{|q|^2}$, hence $v_c^{1/2}$ is the multiplication by $\tfrac{\sqrt{4\pi}}{|q|}$. 
Its real space representation is the convolution by $\tfrac{1}{\sqrt{4\pi}|r-r'|^2}$.
By the Hardy-Littlewood-Sobolev inequality, $v_c^{1/2} : L^2(\RR^3) \to L^{6}(\RR^3)$ is bounded. 

\begin{remark}
Let $\Phi \in \bigwedge\limits_{i=1}^N L^2(\RR^3\times \Z_2)$ and $v \in L^{6}(\RR^3)$, then
\begin{align}
  \braket{\Phi}{Bv} &= \sum\limits_{s_1,\dots,s_N \in \Z_2} \int_{\RR^{3N}} \Phi(x_1\cdots x_N) \Big( \sum\limits_{i=1}^N v(r_i) - \braket{\rho}{v}\Big) \Psi_0(x_1,\dots,x_N) \, \mathrm{d}r_1 \cdots \, \mathrm{d}r_N \\ 
  & = N \sum\limits_{s_1,\dots,s_N \in \Z_2} \int_{\RR^3} v(r) \int_{\RR^{3(N-1)}} \Phi(x,\bar{x}) \Psi_0(x,\bar{x}) \, \mathrm{d}\bar{r} \, \mathrm{d}r - \braket{\rho}{v} \braket{\Phi}{\Psi_0},
\end{align}
where we used the convention $x_i = (r_i,s_i) \in \RR^3 \times \ZZ_2$.
So the adjoint of $B$ is given by
\begin{equation}
  \label{eq:Badjoint}
  B^* \Phi(r_1) = N \sum\limits_{s_1,\dots,s_N \in \Z_2} \int_{\RR^{3(N-1)}} \Phi(x_1,\bar{x}_1) \Psi_0(x_1,\bar{x}_1) \, \mathrm{d}\bar{r}_1 - \braket{\Psi_0}{\Phi} \rho(r_1).
\end{equation}
  Since $B^* \Psi_0 = 0$, one can introduce the operator $H_N^\sharp(v_\mathrm{ext},w) = H_N(v_\mathrm{ext},w) \big|_{\lbrace \Psi_0 \rbrace^\perp}$ and equivalently define $\chi$ as 
  \[
    \chi(\tau) = 2 \, \mathrm{Re} \Big( -\ii \theta(\tau)B^* e^{-\ii \tau (H^\sharp_N(v_\mathrm{ext},w)-E_0)} B\Big) .
  \]
  This observation will be useful to define the Laplace transform of $\chi$ which appears in the expression of the phRPA correlation energy.
\end{remark}

\subsubsection{The Laplace transform of $\chi$}

For a causal operator-valued function $f \in L^\infty(\RR,\cB(\cH))$ (i.e., $f(\tau) = 0$ for $\tau<0$) where $\cH$ is a Hilbert space, a Laplace transform can be defined:
\begin{equation}
  \label{eq:laplace_transform}
  \widetilde{f}(z) = \int_\RR f(\tau) e^{\ii z\tau} \, \mathrm{d}\tau,
\end{equation}
for $z \in \UU = \lbrace z \in \CC | \Im(z) >0 \rbrace$.

A general exposition of the Laplace transform for causal operator-valued functions can be found in~\cite{Cances2016GWpaper}, including a discussion on the conditions under which the Laplace transform converges to the Fourier transform as $z$ goes to the real line.
This type of results are essentially operator versions of Titchmarsh theorem~\cite{titchmarsh}.
For our purposes, the following proposition will be sufficient.

\begin{proposition}[Proposition 13 in \cite{Cances2016GWpaper}]
  \label{prop:laplace_transform}
  Let $A$ be a bounded from below self-adjoint operator on a Hilbert space $\cH$. Let $f(\tau) = -\ii \theta(\tau) e^{-\ii \tau A}$. The Laplace transform of $f$ is given by 
  \begin{equation}
    \widetilde{f}(z) = -(A-z)^{-1},
  \end{equation}
  and defined for all $z \notin \sigma(A)$.
\end{proposition}


We can now deduce the Laplace transform of $\chi$.

\begin{proposition}[Laplace transform of $\chi$]
  The Laplace transform of the retarded linear response operator $\widetilde{\chi}$ defined for $z \notin \sigma(H^\sharp_N(v_\mathrm{ext},w)-E_0^N) \cup \sigma(E_0^N-H^\sharp_N(v_\mathrm{ext},w))$ is given by
  \begin{equation}
    \label{eq:laplace_transform_chi}
    \widetilde{\chi}(z) = - B^*(H^\sharp_N(v_\mathrm{ext},w) -E_0^N - z)^{-1} B-B^*(H^\sharp_N(v_\mathrm{ext},w) -E_0^N + z)^{-1} B.
  \end{equation}
  For $\omega \in \RR, \omega \not= 0$, 
  \begin{equation}
    \widetilde{\chi}(\ii \omega) =-2B^* (H^\sharp_N(v_\mathrm{ext},w)-E_0^N)((H^\sharp_N(v_\mathrm{ext},w)-E_0^N)^2+\omega^2)^{-1} B,
  \end{equation}
  moreover
  \begin{equation}
    \label{eq:remarkable_identity}
    \int_\RR \widetilde{\chi}(\ii \omega) \, \mathrm{d}\omega = - 2\pi B^*B.
  \end{equation}
\end{proposition}

The last equation~\eqref{eq:remarkable_identity} plays an essential part in the derivation of the phRPA correlation energy as it turns out it gives a relationship between $\chi$ and the pair density $\rho^{(2)}$ (see Proposition~\ref{prop:chi_to_pair_density})

\begin{proof}
  Writing 
  \[
    \chi(\tau) = -\ii \theta(\tau) B^* e^{-\ii \tau (H^\sharp_N(v_\mathrm{ext},w)-E_0)} B + \ii\theta(\tau) B^* e^{\ii \tau (H^\sharp_N(v_\mathrm{ext},w)-E_0)} B,
  \]
  and using Proposition~\ref{prop:laplace_transform}, we have \eqref{eq:laplace_transform_chi}.

  The last assertion is proved as follows. Let $P_\lambda$ be the projector-valued spectral measure of $H^\sharp_N(v_\mathrm{ext},w)-E_0^N$. Let $\mu$ be such that $E_0^N <\mu <E_1^N$. Then we have
  \begin{align*}
    \int_\RR \widetilde{\chi}(\ii \omega) \, \mathrm{d}\omega &= - 2B^* \int_\RR \int_\mu^{\infty} \frac{\lambda}{\lambda^2+\omega^2} \, \mathrm{d}P_\lambda \, \mathrm{d}\omega B
     = - 2B^* \int_{\mu}^\infty \int_\RR \frac{\lambda}{\lambda^2+\omega^2}\, \mathrm{d}\omega \, \mathrm{d}P_\lambda B \\
    & = -2\pi B^*(\id-|\Psi_0 \rangle \langle \Psi_0|) B = -2\pi B^*B.
  \end{align*}  
  where we used that $\int_\RR \frac{\alpha}{\alpha^2+\omega^2} \, \mathrm{d}\omega = \pi$ for $\alpha>0$.

\end{proof}

\subsubsection{The noninteracting retarded linear response operator $\chi_0$}

The definition of the retarded linear response operator of a noninteracting system can be simplified.

Let $h = -\frac{1}{2} \Delta + v_0$ with $v_0 \in L^2(\RR^3)+L_\varepsilon^\infty(\RR^3)$ be an operator acting on $L^2(\RR^3)$ with domain $H^2(\RR^3)$. 
We assume that $v_0$ gives rise to at least $n=\tfrac{N}{2}$ negative eigenvalues below the essential spectrum of $h$.
We denote $(\epsilon_i,\phi_i)$ the eigenpairs of $h$.
Additionally, we require an energy gap between the $n$-th and $n+1$-th eigenvalues.\footnote{if the $n+1$-th does not exist, we require a gap between the $n$-th and the bottom of the essential spectrum.}


\begin{assumption}[Energy gap in the one-electron model]
  \label{assumption:energy_gap}
We assume that $\epsilon_n < \epsilon_{n+1}$ (existence of a spectral gap).
\end{assumption}

Let $\Phi_0 = \phi_1\delta_{\uparrow} \wedge \, \phi_1 \delta_{\downarrow} \wedge \cdots \wedge \phi_n \delta_\uparrow \wedge \, \phi_n \delta_{\downarrow}$ be the ground-state of the noninteracting Hamiltonian $H_N(v_0) = H_N({v_0},0)$ with eigenvalue $E_0 = 2 \sum\limits_{i=1}^n \epsilon_i$. 
The electronic density $\rho_0$ is given by $\rho_0(r) = 2\sum\limits_{i=1}^n |\phi_i(r)|^2$. 

The retarded linear response operator $\chi_0$ of $H_N(v_0)$ is 
\begin{equation}
  \tilde{\chi}_0(\ii \omega) = -2B_0^* (H_N(v_0)-E_0^N)[(H_N(v_0)-E_0^N)^2 + \omega^2]^{-1} B_0,
\end{equation}
where $B_0 : L^6(\RR^3) \to \bigwedge_{i=1}^N L^2(\RR^3 \times \ZZ_2)$ is given by  
\[
  B_0v(x_1,\dots,x_N) = \big( \sum\limits_{i=1}^N v(r_i)-\braket{\rho_0}{v} \big) \Phi_0(x_1,\dots,x_N)  
\]
and its adjoint $B_0^* : \bigwedge_{i=1}^N L^2(\RR^3 \times \ZZ_2) \to L^{6/5}(\RR^3)$ by
\[
  B_0^*\Phi(r_1) = N \sum_{s_1, \dots, s_N \in \ZZ_2} \int_{\RR^{3(N-1)}} \Phi_0(x_1,\bar{x}_1) \Phi(x_1,\bar{x}_1) \, \mathrm{d} \bar{r}_1 - \braket{\Phi_0}{\Phi} \rho_0(r_1).
\]

\begin{lemma}
  \label{lem:H_Ntoh}
  Let $P^h_\mu$, $P^{H_N(v_0)}_\lambda$ be respectively the projector-valued spectral measure of $h$ and $H_N(v_0)$. Let $f$ be a bounded continuous function of $\RR$. Then 
  \begin{equation}
    \int_\RR f(\lambda) \, \mathrm{d}P^{H_N(v_0)}_\lambda = \int_{\RR^N} f(\mu_1 + \dots + \mu_N) \, \mathrm{d}P^h_{\mu_1} \otimes \cdots \otimes \mathrm{d}P^h_{\mu_N}. 
  \end{equation}
\end{lemma}

\begin{proof}
  For simplicity, we prove the lemma for $N=2$. Let $A = h \otimes \id$ and $B = \id \otimes \, h$. 
  The operators $A$ and $B$ commute. Denote by $P^A_\mu$ and $P^B_\nu$ their respective projection-valued measures. 
  By definition, $P^A_\mu = P^h_\mu \otimes \id$ and $P^B_\nu = \id \otimes \, P^h_\nu$. 
  Since $A$ and $B$ commute, for $t\in \RR$, we have 
  \begin{equation}
    e^{-\ii t (A+B)} = e^{-\ii t A}e^{-\ii t B},
  \end{equation}
  thus
  \begin{align*}
    \int_\RR e^{-\ii t \lambda} \, \mathrm{d}P^{H_N}_\lambda &= \int_\RR e^{-\ii t \mu} \, \mathrm{d}P^{A}_\mu \int_\RR e^{-\ii t \nu} \, \mathrm{d}P^{B}_\nu= \int_\RR \int_\RR e^{-\ii t (\mu+\nu)} \, \mathrm{d}P^{A}_\mu \mathrm{d}P^{B}_\nu = \int_\RR \int_\RR e^{-\ii t (\mu+\nu)} \, \mathrm{d}P^h_\mu \otimes \mathrm{d}P^h_\nu.
  \end{align*}
\end{proof}

\begin{proposition}[Noninteracting polarizability operator with doubly occupied orbitals]
  \label{prop:chi_zero}
  Let $f \in L^2(\RR^3)$. The noninteracting polarizability operator is given by
  \begin{equation}
    \label{eq:chi0}
    \tilde{\chi}_0(\ii \omega) f(r) = -4\sum\limits_{k=1}^n \phi_k(r) \frac{h-\epsilon_k}{(h-\epsilon_k)^2 + \omega^2}(f\phi_k)(r).
  \end{equation}
\end{proposition}

\begin{proof} Let $(\psi_i)_{1 \leq i \leq N} \in L^2(\RR^3 \times \Z_2)$ defined by $\psi_{2k}(x) = \phi_k(r)\delta_\uparrow(s)$ and $\psi_{2k-1}(x) = \phi_k(r)\delta_\downarrow(s)$. 

  \textbf{Step 1.} We first show that 
  \begin{equation}
    \label{eq:awful1}
    \frac{H_N(v_0)-E_0^N}{(H_N(v_0)-E_0^N)^2 + \omega^2} B_0 f = \frac{1}{\sqrt{N!}} \sum\limits_{j=1}^N \sum\limits_{\sigma \in S_N} (-1)^\sigma \left[\frac{h-\epsilon_{\sigma(j)}}{(h-\epsilon_{\sigma(j)})^2 + \omega^2}f \right](x_j) \prod_{k=1}^N\psi_{\sigma(k)}(x_k).
  \end{equation}

  Since $H_N(v_0)\Phi_0 = E_0^N \Phi_0$, we have
  \begin{equation}
    \frac{H_N(v_0)-E_0^N}{(H_N(v_0)-E_0^N)^2 + \omega^2} B_0 f = \frac{H_N(v_0)-E_0^N}{(H_N(v_0)-E_0^N)^2 + \omega^2} \sum\limits_{i=1}^N f(x_i) \Phi_0(x_1,\dots,x_N).
  \end{equation}
  Applying Lemma~\ref{lem:H_Ntoh}, using that for $\sigma \in S_N$, $E_0 = \sum\limits_{k=1}^N \epsilon_{\sigma(k)}$ we get \eqref{eq:awful1}.

  \textbf{Step 2.} We have
  \begin{align}
    B_0^*& \frac{H_N(v_0)-E_0^N}{(H_N(v_0)-E_0^N)^2 + \omega^2} B_0 f(r) \\
    &= \frac{N}{\sqrt{N!}} \sum_{s,s_2,\dots,s_N \in \mathbb{Z}_2} \int_{\RR^{3(N-1)}} \Phi_0(x,x_2,\dots,x_N) \sum\limits_{\sigma \in S_N} (-1)^\sigma \left[\frac{h-\epsilon_{\sigma(j)}}{(h-\epsilon_{\sigma(j)})^2 + \omega^2}\right] \big(f(r) \psi_{\sigma(1)}(r,s)\big) \nonumber \\
    & \qquad \qquad  \qquad \qquad \prod_{k \geq 2}\psi_{\sigma(k)}(x_k) \, \mathrm{d}r_2 \dots \, \mathrm{d}r_N  \nonumber\\ 
    & \qquad - \braket{\Phi_0}{\frac{1}{\sqrt{N!}} \sum\limits_{j=1}^N \sum\limits_{\sigma \in S_N} (-1)^\sigma \left[\frac{h-\epsilon_{\sigma(j)}}{(h-\epsilon_{\sigma(j)})^2 + \omega^2}f\right] (x_j) \prod_{k=1}^N\psi_{\sigma(k)}(x_k)} \rho_0(r)
  \end{align}
  The second term on the RHS of the above equation vanishes since $(h-\epsilon_{\sigma(j)})\psi_{\sigma(j)}(x_j) = 0$. 
  Thus by orthonormality of $(\psi_i)_{1 \leq i \leq N}$, we obtain:
  \begin{equation}
      \tilde{\chi}_0(\ii \omega) f(r) = -2\sum\limits_{k=1}^N \sum_{s \in \ZZ_2} \psi_k(x) \frac{h-\epsilon_k}{(h-\epsilon_k)^2 + \omega^2}(f\psi_k)(x).
  \end{equation}
  Since $h$ does not act on the spin variable, by definition of $\psi_k$, we obtain Equation~\eqref{eq:chi0}. 
\end{proof}

\begin{remark}
  We remark that our expression \eqref{eq:chi0} is equivalent to the one typically found in the physics and chemistry literature if, as is often assumed in said literature, $h$ is diagonalizable in an orthonormal basis $(\phi_i)_{i \in \NN}$.  Under this assumption,
  \begin{equation}
    \frac{h-\epsilon_k}{(h-\epsilon_k)^2 + \omega^2}(f\phi_k)(x) = \sum\limits_{j=1}^\infty \frac{\epsilon_j-\epsilon_k}{(\epsilon_j-\epsilon_k)^2+\omega^2} \braket{\phi_j \phi_k}{f} \phi_j(x).
  \end{equation}
  Hence, formally the kernel of $\tilde{\chi}_0$ is given by 
  \begin{align}
    \widetilde{\chi}_0(x,y;\ii \omega) &= -2\sum\limits_{k=1}^n \sum\limits_{j=1}^\infty \frac{\epsilon_j-\epsilon_k}{(\epsilon_j-\epsilon_k)^2+\omega^2} \phi_j(y) \phi_k(y) \phi_j(x) \phi_k(x) \\
    &= -2\sum\limits_{k=1}^n \sum\limits_{j=n+1}^\infty \frac{\epsilon_j-\epsilon_k}{(\epsilon_j-\epsilon_k)^2+\omega^2} \phi_j(y) \phi_k(y) \phi_j(x) \phi_k(x) \\
    & = -\sum\limits_{k=1}^n \sum\limits_{j=n+1}^\infty \frac{\phi_j(y) \phi_k(y) \phi_j(x) \phi_k(x)}{\epsilon_j-\epsilon_k+\ii\omega}  + \textnormal{c.c.} \label{eq:chem_chi_0}
  \end{align}
\end{remark}

\begin{remark}
  Like the retarded linear response operator $\widetilde{\chi}$, we notice that $\tilde{\chi}_0$ can equivalently be defined by
  \begin{equation}
    \label{eq:chi0_other_definition}
    \tilde{\chi}_0(\ii \omega) f(r) = -4\sum\limits_{k=1}^n \phi_k(r) P_k \frac{h-\epsilon_k}{(h-\epsilon_k)^2 + \omega^2}P_k(f\phi_k)(r),
  \end{equation}
  where for $1 \leq k \leq n$, $P_k$ is the orthogonal projector onto $\lbrace  \phi_k \rbrace^\perp$. 
\end{remark}

\subsection{Adiabatic connection}

In Kohn-Sham DFT, there is no known useful exact formula for the correlation energy, so we must find some way to approximate it.  One way to derive such approximations is through the adiabatic connection.  The adiabatic connection connects a non-physical non-interacting system to the physical interacting system.  More specifically, let $(F^\lambda)_{0 \leq \lambda \leq 1}$ be the family of operators defined by
\begin{equation}
  F^\lambda[\rho] = F^\lambda_\mathrm{HK}[\rho] + F^\lambda_\mathrm{AC}[\rho],
\end{equation}
where $F^\lambda_\mathrm{HK}[\rho]$ is the constrained search functional
\begin{equation}
  \label{eq:F_HK_lambda_def}
  F^\lambda_\mathrm{HK}[\rho] = \inf_{\Psi \mapsto \rho} \langle \Psi, \big(-\tfrac{1}{2} \Delta + \lambda \sum_{1 \leq i < j \leq N}\frac{1}{|r_i-r_j|}\big) \Psi \rangle,
\end{equation}
and $F^\lambda_\mathrm{AC}$ is defined such that the minimizer $\rho_\lambda$ of
\begin{equation}
  E_0^\lambda = \inf_{\rho \in \mathcal{I}_N} \big( F^\lambda[\rho] + \langle v, \rho \rangle\big),
\end{equation}
is constant in $\lambda$.

By definition, $F_\mathrm{HK}^1$ is simply the Hohenberg-Kohn functional defined in Equation~\eqref{eq:F_HK_definition}, hence $F^1_\mathrm{AC}=0$. 
For $\lambda=0$, $F^0_\mathrm{HK}$ is the Kohn-Sham kinetic energy $T_s$ defined in Equation~\eqref{eq:Ts_definition}, hence by the constraint on the density, $F^0_\mathrm{AC}$ is the Hxc energy $E_\mathrm{Hxc}$ given by \eqref{eq:E_Hxc_definition}. This means by the Hohenberg-Kohn theorem that $E_0^0 = E_0^1$.

\begin{proposition}[Adiabatic connection]
  Let $\Psi^\lambda$ be the minimizer of $F^\lambda_\mathrm{HK}$ in \eqref{eq:F_HK_lambda_def}. Assuming the differentiability of $\lambda \mapsto F_\mathrm{HK}^\lambda$ and $\lambda \mapsto F_\mathrm{AC}^\lambda$, then we have
  \begin{equation}
    \label{eq:Hxc_adiabatic_connection}
    \int_0^1 \tr_{2A} \left(\rho^{(2)}_\lambda V_\mathrm{ee}\right) \dd\lambda = E_\mathrm{Hxc}[\rho],
  \end{equation}
  where $V_\mathrm{ee}$ is the multiplication operator by $\tfrac{1}{|r-r'|}$ and $\rho^{(2)}_\lambda$ is the pair density of $\Psi^\lambda$
  \begin{equation*}
    \rho^{(2)}_\lambda(r_1,r_2) = \binom{N}{2} \sum\limits_{s_1,\dots,s_N \in \Z_2} \int |\Psi^\lambda(r_1s_1,r_2s_2,\dots,r_Ns_N)|^2 \, \mathrm{d}r_3 \dots  \mathrm{d}r_N.
  \end{equation*}
\end{proposition}

\begin{proof}
  By assumption, we have
  \begin{align}
    E^1_0 - E_0^0 &= \int_0^1 \frac{\mathrm{d}E^\lambda_0}{\mathrm{d}\lambda} \, \dd \lambda \\
    &= \int_0^1 \Big\langle \Psi^\lambda, \sum_{1 \leq i < j \leq N}\frac{1}{|r_i-r_j|} \Psi^\lambda \Big\rangle + \frac{\partial F^\lambda_\mathrm{AC}}{\partial\lambda}\, \dd \lambda \\
    & = \int_0^1 \Big\langle \Psi^\lambda, \sum_{1 \leq i < j \leq N}\frac{1}{|r_i-r_j|} \Psi^\lambda \Big\rangle\, \dd \lambda - E_\mathrm{Hxc},
  \end{align}
  where we used that $F^1_\mathrm{AC}=0$ and that $\langle \tfrac{\partial \Psi^\lambda}{\partial \lambda},\big(-\tfrac{1}{2} \Delta + \lambda \sum_{1 \leq i < j \leq N}\frac{1}{|r_i-r_j|}\big) \Psi^\lambda \rangle = 0$.

  Since $E_0^0 = E_0^1$, we have the desired result.
\end{proof}

\subsection{The phRPA correlation energy}

The correlation energy is then defined by substracting the Hartree and exchange energy of the noninteracting system 
\begin{equation}
E_c = \int_0^1 \tr_{2A} \big((\rho^{(2)}_\lambda - \rho^{(2)}_0) V_\mathrm{ee} \big) \, \mathrm{d}\lambda. \label{eq:adiabatic_connection_cor}
\end{equation}
The key to deriving the phRPA correlation energy is to write the correlation energy in terms of $\tilde{\chi}^{(\lambda)}$, the retarded linear response operator derived from ${H}_\lambda$ and then find an appropriate approximation for $\tilde{\chi}^{(\lambda)}$. 

\begin{proposition}
  \label{prop:chi_to_pair_density}
  Let $B$ be the operator defined in Equation~\eqref{eq:B_operator}.
  The operator $(v_c^{1/2})^* B^*B v_c^{1/2} : L^2(\RR^3) \to L^2(\RR^3)$ is trace-class and  
  \begin{equation}
    \label{eq:traceBstarB}
    \tr\Big((v_c^{1/2})^*B^*Bv_c^{1/2}\Big) = 2 \int \frac{\rho^{(2)}(r,r')}{|r-r'|} \, \mathrm{d}r\, \mathrm{d}r',
  \end{equation}
  where $\rho^{(2)}(r,r')$ is the pair density of $\Psi_0$:
  \begin{equation*}
    \rho^{(2)}(r_1,r_2) = \binom{N}{2} \sum\limits_{s_1,\dots,s_N \in \Z_2} \int |\Psi_0(r_1s_1,r_2s_2,\dots,r_Ns_N)|^2 \, \mathrm{d}r_3 \dots  \mathrm{d}r_N.
  \end{equation*}
\end{proposition}

\begin{proof}
  By definition of $B$, for $f \in L^2(\RR^3)$ we have
  \begin{align*}
    B^*Bf(r_1) &= N \sum\limits_{s_1,\dots,s_N \in \Z_2} \int_{\RR^{3(N-1)}} (Bf)(x_1,\bar{x}_1) \Psi_0(x_1,\bar{x}_1) \, \mathrm{d}r_2\dots \mathrm{d}r_N - \braket{\Psi_0}{Bf} \rho_0(r) \\
    & = N \sum\limits_{s_1,\dots,s_N \in \Z_2} \int_{\RR^{3(N-1)}} \Big( \sum\limits_{i=1}^N f(r_i) - \braket{\rho_0}{f} \Big) |\Psi_0(x_1,\bar{x}_1)|^2 \, \mathrm{d}r_2\dots \mathrm{d}r_N \\
    & = N(N-1) \sum\limits_{s_1,\dots,s_N \in \Z_2} \int_{\RR^{3(N-1)}} f(r_2) |\Psi_0(x_1,\bar{x}_1)|^2 \, \mathrm{d}r_2\dots \mathrm{d}r_N \\
    & = 2 \int_{\RR^3} f(r_2) \rho^{(2)}(r_1,r_2) \, \mathrm{d}r_2.
  \end{align*}
  Hence, $(v_c^{1/2})^*B^*Bv_c^{1/2}$ is the operator with kernel $K(r_1,r_2)= \frac{1}{2\pi}  \int \frac{\rho^{(2)}(r,r')}{|r_2-r'|^2 |r_1-r|^2} \, \mathrm{d}r\, \mathrm{d}r'$ which is positive with trace 
  \begin{align}
    \int K(r_1,r_1) \, \mathrm{d}r_1 &= \frac{1}{2\pi}\int \frac{\rho^{(2)}(r,r')}{|r_1-r'|^2 |r_1-r|^2} \, \mathrm{d}r\, \mathrm{d}r' \, \mathrm{d}r_1\\
    &= 2 \int \frac{\rho^{(2)}(r,r')}{|r-r'|} \, \mathrm{d}r\, \mathrm{d}r',
  \end{align}
  where we used that 
  \begin{equation}
    \frac{1}{4\pi}\int_{\RR^3} \frac{1}{|r_1-r'|^2|r_1-r|^2} \, \mathrm{d}r_1 = \frac{1}{4\pi}\frac{1}{|\cdot|^2} \star \frac{1}{|\cdot|^2}(r-r') = \frac{1}{|r-r'|}.
  \end{equation}
\end{proof}

Starting from \eqref{eq:adiabatic_connection_cor}, we can write $E_c$ by using \eqref{eq:traceBstarB}, \eqref{eq:remarkable_identity},
\begin{align*}
  E_c &= \int_0^1 \tr_{2A} \big((\rho^{(2)}_\lambda - \rho^{(2)}_0) V_\mathrm{ee} \big) \, \mathrm{d}\lambda \\
  &= \frac{1}{2} \int_0^1 \tr \big( (v_c^{1/2})^*(B^*_\lambda B_\lambda - B^*_0B_0)v_c^{1/2} \big)\, \mathrm{d}\lambda \\
  &= -\frac{1}{4\pi} \int_\RR \int_0^1 \tr \big( (v_c^{1/2})^*(\tilde{\chi}^{(\lambda)}(\ii \omega) - \tilde{\chi}_\mathrm{KS}(\ii \omega))v_c^{1/2} \big)\, \mathrm{d}\lambda, \numbereq \label{eq:Ec_chi_lambda}
\end{align*}
where $\chi_\mathrm{KS} = \chi^{(0)}$ is the retarded linear response operator of the noninteracting operator ${H}_N(v_\mathrm{KS})$, where $v_\mathrm{KS}$ is the exact Kohn-Sham potential. 
In order to compute \eqref{eq:Ec_chi_lambda}, we would need to know $\tilde{\chi}^{(\lambda)}$ for all $0 \le \lambda \le 1$.  However, finding $\tilde{\chi}^{(\lambda)}$ is essentially as hard as solving the original problem.  Therefore, an approximation is needed, coming in the form of a Dyson equation,
\begin{equation}
  \label{eq:phrpa_dyson_equation}
  \tilde{\chi}^{(\lambda)}(z) = \tilde{\chi}_\mathrm{KS}(z) + \lambda \tilde{\chi}_\mathrm{KS}(z){v}_{c}\tilde{\chi}^{(\lambda)}(z),
\end{equation}
where $v_c$ is the Hartree operator with kernel $\frac{1}{|x-y|}$.  The error in this approximation can be seen more clearly from the point of view of time-dependent density functional theory (TDDFT).  In TDDFT, the exact response function takes the form \cite[Sec.~3.6]{lin2019mathematical}
\begin{equation}
\tilde{\chi}^{(\lambda)}_\text{exact}(z) = \tilde{\chi}_{\mathrm{KS}}(z) - \tilde{\chi}_{\mathrm{KS}}(z)\big(\lambda v_c + f_{xc,\lambda}\big)\tilde{\chi}^{(\lambda)}_\text{exact}(z),
\end{equation}
where $f_{xc,\lambda}$ is the exchange-correlation kernel when the interaction strength is $\lambda$.  By comparing with \eqref{eq:phrpa_dyson_equation}, we see that phRPA is obtained by neglecting the $f_{xc,\lambda}$ term.


The formula for the phRPA correlation energy is finally derived by using the Dyson equation \eqref{eq:phrpa_dyson_equation} and then integrating in $\lambda$,
\begin{align}
  E_c &= -\frac{1}{4\pi}\int_\RR \int_0^1 \tr \big( (\id-\lambda (v_c^{1/2})^*\tilde{\chi}_\mathrm{KS}(\ii \omega)v_c^{1/2})^{-1} (v_c^{1/2})^*\tilde{\chi}_\mathrm{KS}(\ii \omega)v_c^{1/2} - (v_c^{1/2})^*\tilde{\chi}_\mathrm{KS}(\ii \omega)v_c^{1/2} \big)\, \mathrm{d}\lambda \, \mathrm{d} \omega\\
  & = \frac{1}{4\pi}\int_\RR \tr \big( \log(\id- (v_c^{1/2})^*\tilde{\chi}_\mathrm{KS}(\ii \omega)v_c^{1/2}) + (v_c^{1/2})^*\tilde{\chi}_\mathrm{KS}(\ii \omega)v_c^{1/2} \big)\, \mathrm{d} \omega.
\end{align}

\begin{remark}[Other flavors of RPA]
While we have chosen to rewrite $E_c$ in terms of $\chi$, we note that this is not the only choice.  
If we had instead rewritten $E_c$ in terms of the particle-particle Green's function, then we could use a Dyson equation on the particle-particle Green's function, that is similar to \eqref{eq:phrpa_dyson_equation}, to obtain the so-called particle-particle RPA \cite{van2013exchange,van2013exchange}.
\end{remark}

In practice, the exact Kohn-Sham potential is unknown and can only be approximated by some potential $v_0$. 
Hence an additional approximation is introduced by substituing $\chi_\mathrm{KS}$ by the noninteracting retarded linear response operator $\chi_0$ derived from $h=-\tfrac{1}{2} \Delta +v_0$. 
The phRPA correlation energy is thus 
\begin{equation}
  E_c^\mathrm{phRPA}(h) = \frac{1}{4\pi} \int_\mathbb{R} \mathrm{tr}\big( \log(\mathrm{id}-(v_c^{1/2})^*\widetilde{\chi}_0(\ii\omega)v_c^{1/2}) + (v_c^{1/2})^*\widetilde{\chi}_0(\ii \omega)v_c^{1/2} \big) \, \mathrm{d}\omega. \label{eq:defEcRPA}
\end{equation}

The total phRPA energy is then given by 
\begin{multline}
  \label{eq:def_total_phrpa_energy}
  \mathcal{E}^\mathrm{phRPA}(h) = 2\sum_{i=1}^n \langle \phi_i, (-\tfrac{1}{2}\Delta + v_\mathrm{ext}) \phi_i \rangle + \frac{1}{2} \int \frac{\rho(x)\rho(y)}{|x-y|} \, \mathrm{d}x\, \mathrm{d}y \\
  - \int \frac{\big| \sum_{i=1}^n \phi_i(x)\phi_i(y) \big|^2}{|x-y|}\, \mathrm{d}x\, \mathrm{d}y + E_c^{\mathrm{phRPA}}(h),
\end{multline}
where $\rho(x) = 2\sum_{i=1}^n |\phi_i(x)|^2$ and $(\phi_i)_{1\leq i \leq n}$ are the eigenfunctions associated to the $n$ lowest eigenvalues of the operator $h=-\tfrac{1}{2} \Delta + v_0$ where $v_0 = v_\mathrm{ext} + \rho \star \tfrac{1}{|\cdot|} + v_\mathrm{xc}(\rho)$.

\begin{remark}[Time-ordered linear response operator]
  It is also possible to derive the phRPA correlation energy using the time-ordered linear response operator $\chi^T$. This operator is formally defined by 
  \[
    \langle a,\chi^T(t-s)b \rangle = -\ii \langle \Psi_0, \mathcal{T}\lbrace \mathfrak{a}(t),\mathfrak{b}(s) \rbrace \Psi_0 \rangle,
  \]
  where
  \[
    \begin{aligned}
     \mathfrak{a}(t) &= e^{it(H_N(v_\mathrm{ext},w)-E_0)} \Big( \sum_{i=1}^N a(r_i) \Big) e^{-it (H_N(v_\mathrm{ext},w)-E_0)} \\
    \mathfrak{b}(s) & = e^{is(H_N(v_\mathrm{ext},w)-E_0)} \Big( \sum_{i=1}^N b(r_i) \Big) e^{-is (H_N(v_\mathrm{ext},w)-E_0)} 
    \end{aligned}
  \] and $\mathcal{T}$ is the time ordering operator 
  \[
    \mathcal{T}\lbrace A_1(t),A_2(t') \rbrace = \left\lbrace\begin{aligned}
      & A_1(t) A_2(t') , && t>t' \\
      & A_2(t') A_1(t) , && t'>t.
    \end{aligned}\right.
  \]
  Like the retarded linear response function, $\chi^T$ only depends on the time lag $\tau = t-t'$. This expression can be simplified to
  \[
    \chi^T(\tau) = -\ii B^* e^{-\ii (H_N - E_0^N)|\tau|}B.  
  \]
  Splitting $\chi^T$ as 
  \[
    \chi^T(\tau) = -\ii \theta(\tau) B^* e^{-\ii (H_N - E_0^N)\tau}B  -\ii \theta(-\tau) B^* e^{\ii (H_N - E_0^N)\tau}B  ,
  \]
  we can again take the Laplace transform\footnote{for $\theta(-\tau) e^{\ii (H_N - E_0^N)\tau}$, the Laplace transform is originally defined on the lower complex plane $\mathbb{L} = \lbrace z \in \CC \ | \ \mathrm{Im}(z) < 0 \rbrace$ and then extended to $z \notin \sigma(E_0^N-H_N)$.} 
and get that for $z \notin \sigma(H_N^\sharp - E_0^N) \cup \sigma(E_0^N-H_N^\sharp)$,
  \[
    \widetilde{\chi^T}(z) = - B^*(H^\sharp_N -E_0^N - z)^{-1} B-B^*(H^\sharp_N -E_0^N + z)^{-1} B.  
  \]
\end{remark}

\begin{remark}[Self-consistent RPA]
  It is also possible to minimize the total phRPA energy with respect to the input potential $v_0$ \cite{nguyen2014scphrpa,hellgren2012scphrpa}. In practice, this minimization problem is solved using a fixed-point iteration. The question of the well-posedness and the convergence is a challenging mathematical problem.  
\end{remark}

\subsection{Well-posedness of the phRPA correlation formula}
\label{subsec:wellposed}

As the culmination of this section, we prove that, despite the approximation and the assumptions used above, the formula we have just derived is well defined.  The only assumption we need is the existence of an energy gap in the one-electron model.  Without a gap, the system can have a large reaction to an arbitrarily small change in the external potential, in which case, $\chi_0$ is not well defined \cite{weinan2013kohn}.

\begin{theorem}
  \label{thm:phrpa_well-posedness}
  Let $h=-\tfrac{1}{2}\Delta + v_0$ which satisfies the spectral gap Assumption~\ref{assumption:energy_gap}.
  Then the phRPA correlation energy
  \begin{equation}
    \label{eq:phrpa_correlation_energy}
    E_c^\mathrm{phRPA}(h) = \frac{1}{4\pi}\int_\RR \tr  \big( \log(\id- (v_c^{1/2})^* \tilde{\chi}_0(\ii \omega)v_c^{1/2} ) + (v_c^{1/2})^* \tilde{\chi}_0(\ii \omega) v_c^{1/2} \big)\, \mathrm{d} \omega.
  \end{equation}
  is finite.
\end{theorem}

The main idea of the proof consists in using that the inequality $\log(1-x)+x \gtrsim x^2$ for $x>-1$ can be extended to operators and that $(v_c^{1/2})^* \tilde{\chi}_0(\ii \omega)v_c^{1/2}$ is Hilbert-Schmidt. 
Before proving the above theorem, we show a couple of lemmas. 

\begin{lemma}
  \label{lem:trace_of_log_vs_hilbert-schmidt}
  Let $A$ be a Hilbert-Schmidt self-adjoint operator on $L^2(\RR^3)$, such that there is $\alpha <1$ such that for all $f \in L^2(\RR^3)$, we have $\braket{f}{Af} \leq \alpha \|f\|^2$. Then there exists a constant $a>0$ depending on $\alpha$ such that we have 
  \begin{equation}
    -a \|A\|^2_\mathrm{HS}\leq \tr \big( \log(1-A)+A \big) \leq 0.
  \end{equation}
  If $A$ is a nonpositive operator, we can choose $a = \tfrac{1}{2}$.
\end{lemma}

\begin{proof}
  Let $\phi \in L^2(\RR^3)$ and $P_\lambda^A$ be the projector-valued measure of $A$.
  By assumption on $A$, the support of $P_\lambda$ lies in $(-\infty,\alpha)$.
  Thus we have
  \begin{equation}
    \braket{\phi}{(\log(\id -A)+A)\phi} = \int_{-\infty}^\alpha \log(1-\lambda) + \lambda \, \mathrm{d}P^A_{\lambda,\phi}.
  \end{equation}
  There is a constant $a>0$ depending only on $\alpha$ such that $-a\lambda^2 \leq \log(1-\lambda)+\lambda \leq 0$, hence 
  \begin{equation}
    -a \|A\phi\|^2 \leq \braket{\phi}{(\log(\id -A)+A)\phi} \leq 0.
  \end{equation}
  Hence the trace of $\log(\id -A)+A$ is finite and bounded from below by $-a\|A\|_{\mathrm{HS}}^2$.
\end{proof}

\begin{lemma}
  \label{lem:BAB_Hilbert-schmidt}
  Let $A$ be a bounded self-adjoint operator on $L^2(\RR^3)$ and $B$ an operator on $L^2(\RR^3)$ such that $BB^*$ is Hilbert-Schmidt. 
  Then $B^* A B$ is a Hilbert-Schmidt operator with Hilbert-Schmidt norm
  \begin{equation}
    \|B^* A B\|_\mathrm{HS} \leq \|A\| \|BB^*\|_\mathrm{HS}.
  \end{equation}
\end{lemma}

\begin{proof}
  Let $(f_i)_{i \in \N}$ be an orthonormal basis of $L^2(\RR^3)$. By definition we have
  \begin{align*}
    \|B^*AB\|_\mathrm{HS}^2 &= \sum\limits_{i \in \N} \braket{f_i}{B^* ABB^*ABf_i}  
    = \tr \big(B^* ABB^*AB \big) 
    = \tr \big( ABB^*ABB^* \big) \\
    & = \sum\limits_{i \in \N} \braket{f_i}{ABB^*ABB^* f_i} 
    \le \sum_{i \in \NN} \norm{BB^*Af_i} \norm{ABB^*f_i} \\
    &\le \norm{BB^*A}_\textnormal{HS} \norm{ABB^*}_\textnormal{HS}
    \le \norm{A}^2 \norm{BB^*}_\textnormal{HS}^2.
  \end{align*}
\end{proof}

\begin{lemma}
  \label{lem:symmetrized_chi_hilbert_schmidt}
  The symmetrized operator $(v_c^{1/2})^* \tilde{\chi}_0(\ii \omega) v_c^{1/2}$ is Hilbert-Schmidt with a norm bounded by 
  \begin{equation}
    \|(v_c^{1/2})^* \tilde{\chi}_0(\ii \omega) v_c^{1/2}\|_\mathrm{HS} \leq \frac{C}{1+|\omega|}, 
  \end{equation}
  where $C$ depends on $h$ and $N$.
\end{lemma}

\begin{proof}
We want to apply Lemma~\ref{lem:BAB_Hilbert-schmidt} to $(v_c^{1/2})^* \tilde{\chi}_0(\ii \omega) v_c^{1/2}$.
  The operator $P_i \frac{h-\epsilon_i}{(h-\epsilon_i)^2 + \omega^2}P_i$ is self-adjoint and bounded on $L^2(\RR^3)$. Since for $|\omega|>\epsilon_{i+1}-\epsilon_i$, we have
  \begin{equation}
    \sup\limits_{\lambda} \frac{\lambda}{\lambda^2 + \omega^2} = \frac{1}{2|\omega|},
  \end{equation}
  the operator norm is bounded by $\frac{C}{1+|\omega|}$ where $C$ depends on the spectral gap $\epsilon_{i+1} - \epsilon_i$. 

  All that is left to prove is that the operator $(\phi_j v_c^{1/2})(\phi_k v^{1/2}_c)^* = \phi_j v_c \phi_k$ for $1 \leq j,k \leq n$ is a Hilbert-Schmidt operator. 
  The kernel of the operator $\phi_j v_c \phi_k$ is given by 
  \begin{equation}
    K(x,y) := \frac{\phi_j(x)\phi_k(y)}{|x-y|}.
  \end{equation}
  We have 
  \begin{equation}
    \|K\|^2_{L^2} = \int_{\RR^3} \int_{\RR^3} \left( \frac{\phi_j(x)\phi_k(y)}{|x-y|} \right)^2 \, \mathrm{d}x\, \mathrm{d}y \leq C \|\phi_k\|_{H^1}^2 \|\phi_j\|_{L^2}^2  \infty,
  \end{equation}
  by the Hardy inequality. 
  This shows that $(\phi_j v_c^{1/2})(\phi_k v_c^{1/2})^*$ for $1 \leq j,k \leq n$ is a Hilbert-Schmidt operator. 
\end{proof}

We have now all the ingredients to prove the well-posedness of the phRPA correlation energy.

\begin{proof}[Proof of Theorem~\ref{thm:phrpa_well-posedness}]
  The operator $(v_c^{1/2})^* \tilde{\chi}_0(\ii \omega)v_c^{1/2}$ is a nonpositive self-adjoint Hilbert-Schmidt operator on $L^2(\RR^3)$. 
  Hence by Lemma~\ref{lem:trace_of_log_vs_hilbert-schmidt}, for $\omega \not= 0$, we have
  \begin{equation}
    \label{eq:trace_bound}
    - \tfrac{1}{2} \|(v_c^{1/2})^* \tilde{\chi}_0(\ii \omega)v_c^{1/2}\|^2_{\mathrm{HS}} \leq \tr \big( \log(\id- (v_c^{1/2})^* \tilde{\chi}_0(\ii \omega)v_c^{1/2} ) + (v_c^{1/2})^* \tilde{\chi}_0(\ii \omega) v_c^{1/2} \big) \le 0.
  \end{equation}
  By Lemma~\ref{lem:symmetrized_chi_hilbert_schmidt}, we have 
  \begin{equation}
    \|(v_c^{1/2})^* \tilde{\chi}_0(\ii \omega)v_c^{1/2}\|_{\mathrm{HS}} \leq \frac{C}{1+|\omega|}.
  \end{equation}
  Thus the LHS in Equation~\eqref{eq:trace_bound} is integrable with respect to $\omega$, so the phRPA correlation energy is well-defined. 
\end{proof}

\section{Dissociation in the $N$-body model and in restricted Hartree-Fock}
\label{sec:dissociation_Nbody_rhf}

Before showing the exact dissociation of H$_2$ in phRPA, we recall in this section what happens in the $N$-body case and in the restricted Hartree-Fock model. 

\subsection{Exact dissociation in the $N$-body model}

The exact dissociation of H$_2$ is straightforward to establish in the $N$-body model. 
The external potential is given by 
\[
  v_\mathrm{ext}(r) = v(r-R) + v(r+R),
\]
with $v \in L^2(\RR^3)+L_\varepsilon^\infty(\RR^3)$.
The H$_2$ state is described by the ground-state wavefunction $\Psi_0 \in \bigwedge\limits_{i=1}^2 L^2(\RR^3 \times \Z_2)$ of the lowest eigenvalue of 
\begin{equation}
  H_2(v_\mathrm{ext},w) =  \sum\limits_{i=1}^2 \Big(-\frac{1}{2}\Delta_{r_i} + v_{\mathrm{ext}}(r_i) \Big) + w(r_1-r_2),
\end{equation}
where $w(r) = \tfrac{1}{|r|} \in L^2(\RR^3)+L_\varepsilon^\infty(\RR^3)$.

At the dissociation limit, i.e., when $|R|$ goes to $\infty$, each electron will bind to one nucleus, hence we expect the whole system to behave as two independent hydrogen atoms. 

\begin{proposition}[Exact dissociation in the $N$-body electronic Schrödinger equation]
  \label{prop:exact_dissociation_H2_Nbodycase}
  Let $E_0(R)$ be the lowest eigenvalue of $H_2(v_\mathrm{ext},w)$ acting on $\bigwedge_{i=1}^2 L^2(\RR^3 \times \Z_2)$ with domain $\otimes_{i=1}^2 H^2(\RR^3 \times \Z_2)$. 
  Let $\epsilon_0$ be the lowest eigenvalue of $-\frac{1}{2}\Delta + v$. Then we have
  \begin{equation}
    \lim\limits_{|R| \to \infty} E_0(R) = 2 \epsilon_0.
  \end{equation}
\end{proposition}

\begin{proof}[Proof of Proposition~\ref{prop:exact_dissociation_H2_Nbodycase}]
\textbf{Lower bound} Using that $w\geq 0$, we have a lower bound on the ground-state $E_0(R)$
      \begin{align*}
      E_0(\mathrm{H}_2) & \geq 2 \min\limits_{\phi \in L^2(\mathbb{R}^3), \|\phi\|_{L^2}=1} \big\langle \phi, \big( -\tfrac{1}{2} \Delta_r + v(r-R) + v(r+R)\big)  \phi \big\rangle  \\
      & \geq 2 \min\limits_{\phi \in L^2(\mathbb{R}^3), \|\phi\|_{L^2}=1} \big\langle \phi, \big( -\tfrac{1}{2} \Delta_r + v(r) + v(r+2R) \big) \phi \big\rangle.  
    \end{align*}
    Using that $v \in L^2(\RR^3)+L^\infty_\varepsilon(\RR^3)$, we have 
    \[
        \lim\limits_{|R| \to \infty}\min\limits_{\phi \in L^2(\mathbb{R}^3), \|\phi\|_{L^2}=1} \big\langle \phi, \big( -\tfrac{1}{2} \Delta_r + v(r) + v(r+2R) \big) \phi \big\rangle = \min\limits_{\phi \in L^2(\mathbb{R}^3), \|\phi\|_{L^2}=1} \big\langle \phi, \big( -\tfrac{1}{2} \Delta_r + v \big) \phi \big\rangle = \epsilon_0.
    \]
    \textbf{Upper bound}
    Let $\psi(x_1,x_2) = \frac{\phi_0(\cdot-R) \delta_{\uparrow} \wedge \, \phi_0(\cdot+R) \delta_{\downarrow} + \phi_0(\cdot-R) \delta_{\downarrow} \wedge \, \phi_0(\cdot+R) \delta_{\uparrow}}{\|\phi_0(\cdot-R) \delta_{\uparrow} \wedge \, \phi_0(\cdot+R) \delta_{\downarrow} + \phi_0(\cdot-R) \delta_{\downarrow} \wedge \, \phi_0(\cdot+R) \delta_{\uparrow}\|}$ then $E_0(R)$ is bounded by \linebreak $\braopket{\psi}{H_2(v_\mathrm{ext},w)}{\psi}$. Using the exponential decay of the eigenfunction $\phi_0$ and $w\in L^2(\RR^3)+L^\infty_\varepsilon(\RR^3)$, we can show that the upper bound converges to $2\epsilon_0$ as $|R|$ goes to $\infty$. 
\end{proof}

\subsection{Dissociation of H$_2$ in restricted Hartree-Fock}
\label{subsec:dissociation_rhf}

In Section~\ref{sec:phrpa_dissociation}, we will look at the dissociation of H$_2$ in phRPA using the lowest eigenfunctions of a general one-electron model of the H$_2$ molecule. 
In order to prepare for that, we first need to examine what happens for the restricted Hartree-Fock (RHF) energy functional as H$_2$ dissociates.  
We show in Proposition~\ref{prop:dissociation_RHF} that RHF does not correctly dissociate H$_2$.
Instead, it has an error term that is equal to the Hartree energy of a single H atom. 
 
First, we state the properties of our one-electron model that will also be used in Section~\ref{sec:phrpa_dissociation}. 
The one-electron model for the H$_2$ molecule in our analysis is 
\begin{equation}
  \label{eq:one_electron_model_H2}
  h^{(\mathrm{H}_2)} = -\tfrac{1}{2}\Delta +v(r-R)+v(r+R),  
\end{equation}
where $v \in L^2(\RR^3)+L_\varepsilon^\infty(\RR^3)$. 
We assume that $v$ has at least a Coulomb type decay, i.e., for all $r>0$, there exist $v_2 \in L^2(\RR^3)$ and $v_\infty \in L^\infty(\RR^3)$ such that $v=v_2+v_\infty$, $\|v_\infty\|_\infty \leq \frac{1}{r}$ and $\supp v_2 \subset B_r(0)$.

The corresponding Hamiltonian for the H atom is 
\begin{equation}
  \label{eq:one_electron_model_H}
  h^{(\mathrm{H})} = -\tfrac{1}{2}\Delta +v.
\end{equation}
The operator $h^{(\mathrm{H})}$ is acting on $L^2(\RR^3)$ with domain $H^2(\RR^3)$.

\begin{assumption}
  \label{assumption:one-electron-model}
  The lowest eigenvalue of $h^{(\mathrm{H})}$ is negative and simple. 
\end{assumption}

We denote by $(\epsilon^{(\mathrm{H}_2)}_k ,\psi_k)$ the eigenpairs of $h^{(\mathrm{H}_2)}$ and $(\epsilon^{(\mathrm{H})}_k ,\phi_k)$ the eigenpairs of $h^{(\mathrm{H})}$.
Under the above assumptions, we know that the eigenvalue gap $\epsilon_1^{(\mathrm{H}_2)}-\epsilon_0^{(\mathrm{H}_2)}$ is closing, as the eigenfunctions consists of two bubbles located at $R$ and $-R$, which are symmetric for $\psi_0$ and antisymmetric for $\psi_1$. 
These properties are collected in Proposition~\ref{prop:one-e_model_properties}.

\begin{proposition}[Properties of the eigenpairs of $h^{(\mathrm{H})}$ and $h^{(\mathrm{H}_2)}$ \cite{cmp/1103908148}]
  \label{prop:one-e_model_properties}
  The following assertions are true: there are constants $c,C>0$ independent of $R$ such that
  \begin{enumerate}
    \item the eigenfunctions $\phi_k$ have exponential decay
    \item $\big\| \psi_0 - \tfrac{1}{\sqrt{2}}(\phi_0(x-R)+\phi_0(x+R) ) \big\|_{H^1} \leq Ce^{-c|R|}$
    \item $\big\| \psi_1 - \tfrac{1}{\sqrt{2}}(\phi_0(x-R)-\phi_0(x+R) ) \big\|_{H^1} \leq Ce^{-c|R|}$
    \item $|\epsilon^{(\mathrm{H}_2)}_1 - \epsilon^{(\mathrm{H}_2)}_0| \leq Ce^{-c|R|}$
    \item $|\epsilon^{(\mathrm{H})}_0 - \epsilon^{(\mathrm{H}_2)}_0|\leq Ce^{-c|R|}$ 
  \end{enumerate}
\end{proposition}

The RHF energy of the H$_2$ molecule is obtained by restricting the minimization problem for the $N$-body Schrödinger equation to functions of the form 
\begin{equation}
  \Psi(x_1,x_2) = \psi(r_1) \delta_{\uparrow}(s_1) \wedge \, \psi(r_2) \delta_{\downarrow}(s_2).
\end{equation}
In our case, the RHF energy is given by
\begin{equation}
  \label{eq:RHF_energy}
  \mathcal{E}_{\mathrm{H}_2}^{\mathrm{RHF}}(\psi) = 2 \braopket{\psi}{h^{(\mathrm{H}_2)}}{\psi} + \int_{\RR^3 \times \RR^3} |\psi(r)|^2|\psi(r')|^2w(r-r') \, \mathrm{d}r \mathrm{d}r'.
\end{equation}

\begin{proposition}[Dissociation in RHF]
\label{prop:dissociation_RHF}
  Let $\psi_0$ be the ground-state of $h^{(\mathrm{H}_2)}$. Then we have
  \begin{equation}
    \lim\limits_{|R| \to \infty} \mathcal{E}_{\mathrm{H}_2}^{\mathrm{RHF}}(\psi_0) = 2 \braopket{\phi_0}{h^{(\mathrm{H})}}{\phi_0} + \frac{1}{2} \int_{\RR^3 \times \RR^3} |\phi_0(r)|^2 |\phi_0(r')|^2 w(r-r') \, \mathrm{d}r \mathrm{d}r'.
  \end{equation}
\end{proposition}

\begin{proof}
  The proof follows from the exponential localization of the eigenfunctions $\phi_0$ given in Proposition~\ref{prop:one-e_model_properties} and the decay of $w$.
\end{proof}

The Hartree-Fock energy $\mathcal{E}^{\mathrm{RHF}}_{\mathrm{H}}$ of the H atom is simply given by $\epsilon_0$ (since there is only one electron in the system), hence within the Hartree-Fock model, the dissociation limit is incorrectly described.

\begin{remark}
  The wrong extra Hartree energy in the dissociation limit comes from a spurious self-interaction of the doubly occupied state in the RHF model. 
  This occurs because the atomic model has an odd number of electrons and is a consequence of taking into account the spin of the electrons.
  Indeed if the number of electrons was a multiple of 4, the exchange energy in the RHF functional would exactly cancel the Hartree term.
\end{remark}

\section{Exact dissociation of H$_\mathbf{2}$ in phRPA}
\label{sec:phrpa_dissociation}

In this section, we prove the exact dissociation of H$_2$ in phRPA, i.e., the phRPA correlation energy exactly compensates for the extra Hartree energy term in the dissociation limit. 
In our case, the total phRPA energy can be written
\begin{equation}
  \label{eq:total_phrpa_energy_H2}  
  \mathcal{E}^\mathrm{phRPA}(H_2) = \mathcal{E}_{\mathrm{H}_2}^\mathrm{RHF}(\psi_0)+E_c^\mathrm{phRPA}(h^{(\mathrm{H}_2)}),
\end{equation}
where $\psi_0$ is the ground-state of $h^{(\mathrm{H}_2)}$ and $E_c^\mathrm{phRPA}(h^{(\mathrm{H}_2)})$ is the phRPA correlation energy given by the linear response operator of $h^{(\mathrm{H}_2)}$.
Likewise the total phRPA energy of a single $H$ atom is given by 
\begin{equation}
  \label{eq:total_phrpa_energy_H}  
  \mathcal{E}^\mathrm{phRPA}(H) = \mathcal{E}_{\mathrm{H}}^\mathrm{RHF}(\psi_0)+E_c^\mathrm{phRPA}(h^{(\mathrm{H})}) = \langle \phi_0, h^{(\mathrm{H})} \phi_0 \rangle + E_c^\mathrm{phRPA}(h^{(\mathrm{H})}),
\end{equation}
where $\phi_0$ is ground-state of $h^{(\mathrm{H})}$ defined in Equation~\eqref{eq:one_electron_model_H}.

\begin{theorem}[Exact dissociation in the phRPA model]
  \label{thm:phrpa_dissociation}
  Let $\mathcal{E}^\mathrm{phRPA}(H_2)$ and $\mathcal{E}^\mathrm{phRPA}(H)$ be respectively the total phRPA energy of H$_2$ defined in Equation~\eqref{eq:total_phrpa_energy_H2} and the total phRPA energy of H defined in Equation~\eqref{eq:total_phrpa_energy_H}.
  Then under Assumption~\ref{assumption:one-electron-model}, we have exact dissociation of H$_2$ \emph{i.e.}
  \[
    \lim\limits_{|R| \to \infty} \mathcal{E}^\mathrm{phRPA}(H_2) = 2\, \mathcal{E}^\mathrm{phRPA}(H).
  \]
\end{theorem}

From Section~\ref{subsec:dissociation_rhf}, it is sufficient to show that the phRPA correlation energy cancels the spurious term in Proposition~\ref{prop:dissociation_RHF}
  \begin{equation}
    \lim_{|R| \to \infty} E_c^\mathrm{phRPA}(H_2) = -\frac{1}{2} \int_{\RR^3 \times \RR^3} |\phi_0(r)|^2 |\phi_0(r')|^2 w(r-r') \, \mathrm{d}r \mathrm{d}r' + 2 E_c^\mathrm{phRPA}(H),
  \end{equation}
The idea of the proof goes as follows. By inserting a resolution of identity $P=\mathrm{id}-|\psi_0\rangle\langle \psi_0|-|\psi_1\rangle\langle \psi_1|$, we write the linear response operator $\chi_0$ of $h^{(\mathrm{H}_2)}$ as 
\[
  \widetilde{\chi_0}(\ii \omega) =  \frac{\epsilon_1^{(\mathrm{H}_2)}-\epsilon_0^{(\mathrm{H}_2)}}{(\epsilon_1^{(\mathrm{H}_2)}-\epsilon_0^{(\mathrm{H}_2)})^2 + \omega^2} |\psi_0 \psi_1 \rangle \langle \psi_0 \psi_1 | + \psi_0 P \frac{h^{(\mathrm{H}_2)}-\epsilon^{(\mathrm{H}_2)}_0}{(h^{(\mathrm{H}_2)}-\epsilon^{(\mathrm{H}_2)}_0)^2 + \omega^2} P \psi_0,
\]
where $|\psi_0 \psi_1 \rangle \langle \psi_0 \psi_1 |$ is the projection onto the product $\psi_0 \psi_1 \in L^2(\RR^3)$ by Proposition~\ref{prop:one-e_model_properties}. 
Plugged in the phRPA correlation energy, the rank-one term gives the first two terms in Proposition~\ref{prop:splitting_phra_Ec}. The first term converges to the negative of the Hartree energy that cancels the RHF spurious term. 
Finally, in Section~\ref{subsec:dissociation-trace-remainder}
we show that the remainder splits into twice the phRPA correlation energy of a single H atom. 
This results from the locality of the ground-state $\phi_0$ and of the resolvent (see Lemma~\ref{lem:locality_greens_function}). 
The estimation of the vanishing terms requires bounds that need to be traceable and integrable with respect to $\omega$. 
In this regard, Lemma~\ref{lem:log_lemma} plays a key role to in order to estimate the trace of operators $\log(\mathrm{id}-A)+A$ by the Hilbert-Schmidt norm of $A$.

\begin{remark}
  The proof of the dissociation limit can be extended for other diatomic molecules under the assumption that $n$ is odd and the highest occupied state with energy $\epsilon_n$ is simple.
\end{remark}

\subsection{The splitting of the trace}

\begin{proposition}[Splitting of the correlation energy]
  \label{prop:splitting_phra_Ec}
  Let $P$ be the orthogonal projector defined by $P = \id-|\psi_0 \rangle \langle \psi_0 | -|\psi_1 \rangle \langle \psi_1 |$.  Let $K(\omega)$ be the operator defined by 
  \begin{equation}
    \label{eq:reduced_chi}
    K(\omega) = 4 (v_c^{1/2})^* \psi_0 P \frac{h^{(\mathrm{H}_2)}-\epsilon^{(\mathrm{H}_2)}_0}{(h^{(\mathrm{H}_2)}-\epsilon^{(\mathrm{H}_2)}_0)^2 + \omega^2} P \psi_0 v_c^{1/2}.
  \end{equation}
  Then the phRPA correlation energy defined in \eqref{eq:phrpa_correlation_energy} can be written as a sum of three terms 
  \begin{multline}
    \label{eq:splitting_trace}
    E_c^{\mathrm{phRPA}}(H_2) = -\frac{1}{\pi} \int_\RR \frac{\epsilon_1^{(\mathrm{H}_2)}-\epsilon_0^{(\mathrm{H}_2)}}{(\epsilon_1^{(\mathrm{H}_2)}-\epsilon_0^{(\mathrm{H}_2)})^2 + \omega^2} \, \mathrm{d}\omega  \int_{\RR^3 \times \RR^3} \psi_1(r)\psi_0(r) \psi_1(r')\psi_0(r') w(r-r') \, \mathrm{d}r \mathrm{d}r' \\
    + \frac{1}{4\pi} \int_\RR \log \Big( 1+ \frac{\epsilon_1^{(\mathrm{H}_2)}-\epsilon_0^{(\mathrm{H}_2)}}{(\epsilon_1^{(\mathrm{H}_2)}-\epsilon_0^{(\mathrm{H}_2)})^2 + \omega^2} \braopket{v_c^{1/2}\psi_1 \psi_0}{(\id+K(\omega))^{-1}}{v_c^{1/2}\psi_1 \psi_0} \Big)\, \mathrm{d}\omega \\
    + \frac{1}{4\pi} \int_\RR \tr \Big( \log\big( \id + K(\omega) \big) - K(\omega) \Big) \, \mathrm{d}\omega.
  \end{multline}
\end{proposition}

The phRPA correlation energy splits into three terms
\begin{itemize}
  \item the first one cancels the extra RHF term in the dissociation limit
  \item the second term goes to zero because the gap $\epsilon_1^{(\mathrm{H}_2)}-\epsilon_0^{(\mathrm{H}_2)}$ closes
  \item in the limit, the remainder gives twice the phRPA correlation energy of an H atom.
\end{itemize}

The proofs of these statements can be found in Section~\ref{subsec:limits-rank-one-terms} and Section~\ref{subsec:dissociation-trace-remainder}.
Before proving Proposition~\ref{prop:splitting_phra_Ec}, we state two useful lemmas.

\begin{lemma}[Shermann-Morrison formula]
  \label{lem:shermann-morrison}
  Let $A$ be a bounded, nonnegative, self-adjoint operator on $L^2(\RR^3)$, $\alpha \geq 0$ and $\zeta \in L^2(\RR^3)$. Let $A_\alpha = A + \alpha |\zeta \rangle \langle \zeta |$. For $z \in \mathbb{C} \setminus \R_+$, we have
  \begin{equation}
    (A_\alpha - z)^{-1} = (A-z)^{-1} - \frac{\alpha}{1+\alpha \braket{\zeta}{(A-z)^{-1}\zeta}} |(A-z)^{-1}\zeta \rangle \langle (A-z)^{-1}\zeta |.
  \end{equation}
\end{lemma}

\begin{proof}
  Let $z \in \mathbb{C} \setminus \R_+$. Since $\alpha \geq 0$ and $A$ is nonnegative and self-adjoint, $z$ is not in the spectrum of $A$ or $A_\alpha$. By the second resolvent identity, we have
  \begin{align}
    (A_\alpha - z)^{-1} & = (A+\alpha |\zeta \rangle \langle \zeta |-z)^{-1} \\
    &= (A-z)^{-1} \Big( \id+\alpha |\zeta \rangle \langle (A-z)^{-1}\zeta | \Big)^{-1}.
  \end{align}
  The inverse of $\id+\alpha |\zeta \rangle \langle (A-z)^{-1}\zeta|$ is $\id- \frac{\alpha}{1+\alpha \braket{\zeta}{(A-z)^{-1}\zeta}} |\zeta \rangle \langle (A-z)^{-1}\zeta|$. Inserting this in the previous expression finishes the proof of the lemma.
\end{proof}

\begin{lemma}
  \label{lem:log_identity}
  For $t >-1$, we have
  \begin{equation}
    \log(1+t) = \int_0^\infty \frac{1}{1+s}-\frac{1}{1+s+t} \, \mathrm{d}s.
  \end{equation}
\end{lemma}

\begin{proof}
  For $S>0$, we have 
  \begin{align}
    \int_0^S \frac{1}{1+s}-\frac{1}{1+s+t} \, \mathrm{d}s = \log(1+S)-\log(1+t+S) + \log(1+t) \underset{S \to \infty}{\longrightarrow} \log(1+t).
  \end{align}
\end{proof}

We have all the elements to prove Proposition~\ref{prop:splitting_phra_Ec}.

\begin{proof}[Proof of Proposition~\ref{prop:splitting_phra_Ec}]
  Using Lemma~\ref{lem:shermann-morrison} and Lemma~\ref{lem:log_identity} with $\alpha = \frac{\epsilon_1^{(\mathrm{H}_2)}-\epsilon_0^{(\mathrm{H}_2)}}{(\epsilon_1^{(\mathrm{H}_2)}-\epsilon_0^{(\mathrm{H}_2)})^2 + \omega^2} $, $\zeta = v_c^{1/2} \psi_1 \psi_0$ and $A = K(\omega)$, we show that $\log(\id+A_\alpha) - \log(\id+A)$ is a rank-one operator. Let $P_\lambda^A$ and $P_\mu^{A_\alpha}$ be respectively the operator-valued measure of $A$ and $A_\alpha$. Then we have
  \begin{align}
    \log(\id+A_\alpha) - \log(\id+A) & = \int_0^\infty \log(1+\mu) \, \mathrm{d}P^{A_\alpha}_\mu - \int_0^\infty \log(1+\lambda) \, \mathrm{d}P^{A}_\lambda \\
    &= \int_0^\infty \int_0^\infty (1+t)^{-1} - (1+\mu+t)^{-1} \, \mathrm{d}t  \, \mathrm{d}P^{A_\alpha}_\mu \\
    & \qquad \qquad \qquad - \int_0^\infty\int_0^\infty  (1+t)^{-1} - (1+\lambda+t)^{-1} \, \mathrm{d}t \, \mathrm{d}P^{A}_\lambda \\
    &= \int_0^\infty (1+t+A)^{-1} - (1+t+A_\alpha)^{-1} \, \mathrm{d}t \\
    &= \int_0^\infty \frac{\alpha}{1+\alpha \braket{\zeta}{(A+1+t)^{-1}\zeta}} |(A+1+t)^{-1}\zeta \rangle \langle (A+1+t)^{-1}\zeta| \, \mathrm{d}t . 
  \end{align}
  Hence $\log(\id+A_\alpha) - \log(\id+A)$ is traceable and we have
  \begin{align}
    \tr \big( \log(\id+A_\alpha) - \log(\id+A) \big) &= \alpha \int_0^\infty \frac{\braket{\zeta}{(A+1+t)^{-2}\zeta}}{1+\alpha \braket{\zeta}{(A+1+t)^{-1}\zeta}} \, \mathrm{d}t \\
    &= \log \big( 1+\alpha \braket{\zeta}{(A+1)^{-1}\zeta} \big)\label{eq:log_term}.
  \end{align}
  By definition of $A$ and $A_\alpha$, the trace of the difference is given by
  \begin{equation}
    \tr\big( A_\alpha - A \big) = - \alpha \|\zeta\|^2. \label{eq:RHF_canceled_term}
  \end{equation}
  Combining Equations~\eqref{eq:log_term} and \eqref{eq:RHF_canceled_term} with the formula for the phRPA correlation energy~\eqref{eq:phrpa_correlation_energy}, we obtain Equation~\eqref{eq:splitting_trace}.
\end{proof}

\subsection{Limits of the rank-1 terms}
\label{subsec:limits-rank-one-terms}

In this section, we take the limit as $|R|\to\infty$ of the first two terms in \eqref{eq:splitting_trace}.

\begin{proposition}
  \label{prop:limits_rank_one_terms}
  Using notation introduced in Proposition~\ref{prop:splitting_phra_Ec}, we have 
  \begin{multline}
    \lim_{|R|\to \infty} -\frac{1}{\pi} \int_\RR \frac{\epsilon_1^{(\mathrm{H}_2)}-\epsilon_0^{(\mathrm{H}_2)}}{(\epsilon_1^{(\mathrm{H}_2)}-\epsilon_0^{(\mathrm{H}_2)})^2 + \omega^2} \, \mathrm{d}\omega  \int_{\RR^3 \times \RR^3} \psi_1(r)\psi_0(r) \psi_1(r')\psi_0(r') w(r-r') \, \mathrm{d}r \mathrm{d}r' \\
    = -\frac{1}{2} \int_{\RR^3\times \RR^3} |\phi_0(r-R)|^2 |\phi_0(r-R)|^2 w(r,r') \dd r \dd r', 
  \end{multline}
  and 
  \begin{equation}
    \lim_{|R|\to \infty} \int_\RR \log \left( 1+ \frac{\epsilon_1^{(\mathrm{H}_2)}-\epsilon_0^{(\mathrm{H}_2)}}{(\epsilon_1^{(\mathrm{H}_2)}-\epsilon_0^{(\mathrm{H}_2)})^2 + \omega^2} \braopket{v_c^{1/2}\psi_1 \psi_0}{(\id+K(\omega))^{-1}}{v_c^{1/2}\psi_1 \psi_0} \right)\, \mathrm{d}\omega = 0.
  \end{equation}
\end{proposition}

The second limit is surprising as linearizing the logarithm would give the same expression as in the first limit.

\begin{proof}
For the first term, we notice that we can integrate in $\omega$ for any finite value of $R$ by using $\int_\RR \frac{a}{a^2+\omega^2} \dd\omega = \pi$ for any $a > 0$.  Then we split the $\psi$'s using Prop.~\ref{prop:one-e_model_properties}.  Doing these, the first term becomes
\begin{align*}
&-\frac{1}{\pi} \int_\RR \frac{\epsilon_1^{(\mathrm{H}_2)}-\epsilon_0^{(\mathrm{H}_2)}}{(\epsilon_1^{(\mathrm{H}_2)}-\epsilon_0^{(\mathrm{H}_2)})^2 + \omega^2} \, \mathrm{d}\omega  \int_{\RR^3 \times \RR^3} \psi_1(r)\psi_0(r) \psi_1(r')\psi_0(r') w(r-r') \, \mathrm{d}r \mathrm{d}r' \\
&\hspace{10mm}= -\frac{1}{4} \int_{\RR^3\times \RR^3} \left(\phi_0(r-R)+\phi_0(r+R)\right)\left(\phi_0(r-R)-\phi_0(r+R)\right) \\
&\hspace{30mm}\cdot\left(\phi_0(r'-R)+\phi_0(r'+R)\right)\left(\phi_0(r'-R)-\phi_0(r'+R)\right) w(r,r') \dd r \dd r' + \cO(e^{-cR}) \\
&\hspace{10mm}= -\frac{1}{4} \int_{\RR^3\times \RR^3} \left(|\phi_0(r-R)|^2 +|\phi_0(r+R)|^2\right)\left(|\phi_0(r'-R)|^2 + |\phi_0(r'+R)|^2\right) w(r,r') \dd r \dd r' + \cO(e^{-cR}) \\
&\hspace{10mm}= -\frac{1}{2} \int_{\RR^3\times \RR^3} |\phi_0(r-R)|^2 |\phi_0(r-R)|^2 w(r,r') \dd r \dd r' \\
&\hspace{20mm}+ \frac{1}{2} \int_{\RR^3\times \RR^3} |\phi_0(r-R)|^2 |\phi_0(r+R)|^2 w(r,r') \dd r \dd r' + \cO(e^{-cR})\\
&\hspace{10mm}= -\frac{1}{2} \int_{\RR^3\times \RR^3} |\phi_0(r-R)|^2 |\phi_0(r-R)|^2 w(r,r') \dd r \dd r' + \cO(R^{-1}), \numbereq
\end{align*}
as promised.

For the second term in \eqref{eq:splitting_trace}, first note that for all $R$ and $\omega$, $K(\omega)$ is nonnegative.  Therefore, $0 \le (\id+K(\omega))^{-1} \le \id$ and
\begin{equation}
0 \le \braopket{v_c^{1/2}\psi_1 \psi_0}{(\id+K(\omega))^{-1}}{v_c^{1/2}\psi_1 \psi_0} \le c,
\end{equation}
for some constant $c$ independent of $R$ and $\omega$.  Let $g_R = \epsilon_1^{(\mathrm{H}_2)}-\epsilon_0^{(\mathrm{H}_2)}$ be the HOMO--LUMO energy gap.  Since $g_R > 0$, the second term in \eqref{eq:splitting_trace} can be bounded as
\begin{align*}
0 & \le \frac{1}{2\pi} \int_\RR \log \left( 1+ \frac{g_R}{g_R^2 + \omega^2} \braopket{v_c^{1/2}\psi_1 \psi_0}{(\id+K(\omega))^{-1}}{v_c^{1/2}\psi_1 \psi_0} \right)\, \mathrm{d}\omega \\
  &\le \frac{1}{2\pi} \int_\RR \log \left( 1+ \frac{cg_R}{g_R^2 + \omega^2} \right)\, \mathrm{d}\omega. \numbereq
\end{align*}
Next, we show that this upper bound goes to 0 as $R\to \infty$ (note that $g_R \to 0$ as $R\to\infty$).  To do this, we use integration by parts and then the residue theorem:
\begin{align*}
&\frac{1}{2\pi} \int_\RR \log \left( 1+ \frac{cg_R}{g_R^2 + \omega^2} \right)\, \mathrm{d}\omega = \underbrace{\frac{1}{2\pi} \omega \log\left(1+\frac{cg_R}{g_R^2+\omega^2}\right)\bigg|_{-\infty}^\infty}_{=\,0} + \frac{cg_R}{\pi} \int_\RR \frac{\omega^2}{(\omega^2 + g_R(g_R+c))(\omega^2 + g_R^2)} \dd\omega \\
  &= 2\pi i \frac{cg_R}{\pi} \left[ \text{Res}\left(\frac{\omega^2}{(\omega^2 + g_R(g_R+c))(\omega^2 + g_R^2)}, ig\right) + \text{Res}\left(\frac{\omega^2}{(\omega^2 + g_R(g_R+c))(\omega^2 + g_R^2)}, i\sqrt{g(g+c)}\right)\right] \\
  &= g_R\left(\sqrt{1+\frac{c}{g_R}}-1\right),
\end{align*}
which goes to zero as $R \to \infty$.
\end{proof}

\subsection{Dissociation of the remainder}
\label{subsec:dissociation-trace-remainder}

The idea is to break $K(\omega)$ defined in Equation~\eqref{eq:reduced_chi} in operators whose products go to 0, to use Lemma~\ref{lem:log_lemma}.

\begin{lemma}
  \label{lem:reduced_chi_split}
  Let $K(\omega)$ be the operator defined in Equation~\eqref{eq:reduced_chi}. Let $\Pi_R = \id - |\phi_0^R \rangle  \langle  \phi_0^R |$ and $\Pi_{-R} = \id - |\phi_0^{-R} \rangle  \langle  \phi_0^{-R}|$. We have
  \begin{multline}
    \label{eq:reduced_chi_split}
    \tfrac{1}{2} K(\omega) = (v_c^{1/2})^* \phi_0^R \Pi_R P \frac{h-\epsilon_0^{(\mathrm{H}_2)}}{(h-\epsilon_0^{(\mathrm{H}_2)})^2 + \omega^2} P \Pi_R \phi_0^R v_c^{1/2} 
    + (v_c^{1/2})^* \phi_0^{-R} \Pi_{-R} P \frac{h-\epsilon_0^{(\mathrm{H}_2)}}{(h-\epsilon_0^{(\mathrm{H}_2)})^2 + \omega^2} P \Pi_{-R} \phi_0^{-R} v_c^{1/2} \\
    + (v_c^{1/2})^* \phi_0^R \Pi_R P \frac{h-\epsilon_0^{(\mathrm{H}_2)}}{(h-\epsilon_0^{(\mathrm{H}_2)})^2 + \omega^2} P \Pi_{-R} \phi_0^{-R} v_c^{1/2} 
    + (v_c^{1/2})^* \phi_0^{-R} \Pi_{-R} P \frac{h-\epsilon_0^{(\mathrm{H}_2)}}{(h-\epsilon_0^{(\mathrm{H}_2)})^2 + \omega^2} P \Pi_R \phi_0^R v_c^{1/2} \\
    + \mathscr{R}(\omega) ,
  \end{multline}
  where $\mathscr{R}(\omega)$ is a Hilbert-Schmidt operator on $L^2(\RR^3)$ with 
  \begin{equation}
    \|\mathscr{R}(\omega)\|_{HS} \leq \frac{Ce^{-cR}}{1+\omega^2},
  \end{equation}
  for some constants $c,C>0$ independent of $R$ and $\omega$.
\end{lemma}

\begin{proof} 
  Let $f \in L^2(\RR^3)$. We have
  \begin{align}
    P\psi_0 f = \frac{1}{\sqrt{2}} P(\phi_0^R + \phi_0^{-R})f + P\big(\psi_0 - \tfrac{1}{\sqrt{2}}(\phi_0^R + \phi_0^{-R}) \big)f.
  \end{align}
  By Proposition~\ref{prop:one-e_model_properties}, we have 
  \begin{equation}
    |\psi_0\rangle \langle\psi_0 |+|\psi_1\rangle \langle\psi_1 | = |\phi_0^R \rangle \langle \phi_0^R| + |\phi_0^{-R} \rangle \langle \phi_0^{-R}| + P\cO(e^{-c|R|}),
  \end{equation}
  and using the definition of $P$, we get
  \begin{align}
    P(\phi_0^R + \phi_0^{-R})f &= \big( \id - |\phi_0^R \rangle \langle \phi_0^R| - |\phi_0^{-R} \rangle \langle \phi_0^{-R}| \big) (\phi_0^R + \phi_0^{-R})f+ P\cO(e^{-c|R|}) \\
    &=\Pi_R \phi_0^Rf + \Pi_{-R} \phi_0^{-R}f + P \cO (e^{-c|R|}).
  \end{align}
  Inserting this expression in $K(\omega)$, we obtain Equation~\eqref{eq:reduced_chi_split}.
\end{proof}

\begin{lemma}
  \label{lem:log_lemma}
  Let $A$ and $B$ be bounded self-adjoint operators such that $A, B, A+B \ge -\frac{1}{2}$. Then,
  \begin{equation}
  \|\log(\id+A+B) - \log(\id+A) - \log(\id+B)\| \leq C\norm{AB},
  \end{equation}
  for some constant $C$ independent of $A$ and $B$.
\end{lemma}

\begin{proof}
  Using the identity in Lemma~\ref{lem:log_identity}, we have
  \begin{equation}
    \log(\id + A+B) - \log(\id +A)- \log(\id +B) = \int_0^\infty (1+t+A)^{-1} + (1+t+B)^{-1} - (1+t)^{-1} - (1+t+A+B)^{-1} \, \mathrm{d}t.
  \end{equation}
  Using the resolvent identity $(1+t+C)^{-1} = (1+t)^{-1}-(1+t)^{-1}C(1+t+C)^{-1}$ for $C = A,B,A+B$, we get
  \begin{multline}
    \log(\id + A+B) - \log(\id +A)- \log(\id +B)\\
     = -\int_0^\infty (1+t)^{-1}B(1+t+B)^{-1}+ (1+t)^{-1}A(1+t+A)^{-1}- (1+t)^{-1}(A+B)(1+t+A+B)^{-1} \, \mathrm{d}t.
  \end{multline}
  Applying again a resolvent identity, we obtain 
  \begin{align}
    \log(&\id+A+B) - \log(\id +A)- \log(\id +B) \nonumber \\
    &= -\int_0^\infty (1+t)^{-1}B(1+t+B)^{-1}+ (1+t)^{-1}A(1+t+A)^{-1}
     - (1+t)^{-1}(A+B)(1+t+A+B)^{-1} \, \mathrm{d}t \\
    &= -\int_0^\infty (1+t)^{-1}A(1+t+A)^{-1}B(1+t+A+B)^{-1} 
     + (1+t)^{-1}B(1+t+B)^{-1}A(1+t+A+B)^{-1} \, \mathrm{d}t \\
    &= -\int_0^\infty (1+t)^{-1}(1+t+A)^{-1}AB(1+t+A+B)^{-1} 
     + (1+t)^{-1}(1+t+B)^{-1}BA(1+t+A+B)^{-1} \, \mathrm{d}t.
  \end{align}
  Using that $\norm{AB} = \norm{(AB)^*} = \norm{BA}$, we conclude that there is constant $C>0$ such that
  \begin{equation}
    \big\|\log(\id+A+B) - \log(\id +A)- \log(\id +B)\big\| \leq C \norm{AB}.
  \end{equation}
\end{proof}

%

\begin{lemma}
  \label{lem:first_piece_third_term}
  Let $K_R(\omega)$ be the operator defined by 
  \begin{equation}
    \label{eq:K_R}
    K_R(\omega) =  2(v_c^{1/2})^* \phi_0^R \Pi_R P \frac{h^{(\mathrm{H}_2)}-\epsilon_0^{(\mathrm{H}_2)}}{(h^{(\mathrm{H}_2)}-\epsilon_0^{(\mathrm{H}_2)})^2 + \omega^2} P \Pi_R \phi_0^R v_c^{1/2}.
  \end{equation}
  We have
  \begin{equation}
    \label{eq:final_step_in_Ec_phrpa}
    \int_\RR \tr \Big( \log\big( \id + K(\omega) \big) - K(\omega) \Big) \, \mathrm{d}\omega = 2 \int_\RR \tr\Big( \log\big( \id + K_R(\omega) \big) - K_R(\omega) \Big) \, \mathrm{d}\omega+ \cO \Big( \tfrac{1}{|R|^2} \Big).
  \end{equation} 
\end{lemma}

\begin{proof}
  By Proposition~\ref{prop:one-e_model_properties} we notice that 
  \begin{equation}
    \| (v_c^{1/2}\phi_0^R)^* v_c^{1/2} \phi^{-R}_0 \|^2_{HS} = \int_{\RR^3 \times \RR^3} \frac{|\phi_0(r-R)|^2|\phi_0(r'+R)|^2}{|r-r'|^2} \, \mathrm{d}r \, \mathrm{d}r' = \cO \big(  \tfrac{1}{|R|^2}\big).
  \end{equation}
  We can thus combine Lemma~\ref{lem:reduced_chi_split} with \ref{lem:log_lemma} to split $\tr \Big( \log\big( \id + K(\omega) \big) - K(\omega) \Big)$. Using Lemma~\ref{lem:trace_of_log_vs_hilbert-schmidt}, we conclude that 
  \begin{equation}
    \int_\RR \tr \Big( \log\big( \id + K(\omega) \big) - K(\omega) \Big) \, \mathrm{d}\omega = 2 \int_\RR \tr\Big( \log\big( \id + K_R(\omega) \big) - K_R(\omega) \Big) \, \mathrm{d}\omega+ \cO \Big( \tfrac{1}{|R|^2} \Big).
  \end{equation} 
\end{proof}

All that is left to prove is that the limit of the right hand side in Equation~\eqref{eq:final_step_in_Ec_phrpa} converges to the correlation energy of a single H atom. 
This is a consequence of the locality of the Green's function and the exponential localization of the function $\phi_0^R$. 
In our case, however, when deriving the locality of the Green's function, we need to obtain a bound that is integrable with respect to $\omega$.

\begin{lemma}[Locality of the Green's function]
  \label{lem:locality_greens_function}
  Let $\eta_1$ and $\eta_2$ be smooth cut-off functions with disjoint support. Let 
  \[
   R \leq \min(\dist(\supp \eta_1,0),\dist(\supp \eta_2,0),\dist(\supp \eta_1,\supp \eta_2)) 
  \] and $\omega \not=0$. Let $\Pi = \id - |\phi_0\rangle \langle \phi_0 |$.
  The operator $\eta_1 \Pi(h^{(\mathrm{H})} - \epsilon_0^{(\mathrm{H})}+\ii \omega)^{-1}\eta_2$ is a bounded operator from $L^1(\RR^3)$ to $L^2(\RR^3)$ with operator norm bounded above by $\frac{C}{R(g+|\omega|)}$, for some positive constant $C$ independent of $R$ and $\omega$, where $g$ is the spectral gap of $h^{(\mathrm{H})}$.
\end{lemma}

\begin{proof}
  By the second resolvent identity, we have
  \begin{equation}
    \eta_1 \Pi(h^{(\mathrm{H})} - \epsilon_0^{(\mathrm{H})}+\ii \omega)^{-1}\eta_2 = \eta_1 \Pi(-\tfrac{1}{2}\Delta - \epsilon_0^{(\mathrm{H})}+\ii \omega)^{-1}\eta_2 - \eta_1 \Pi(h^{(\mathrm{H})} - \epsilon_0^{(\mathrm{H})}+\ii \omega)^{-1} v (-\tfrac{1}{2}\Delta - \epsilon_0^{(\mathrm{H})}+\ii \omega)^{-1}\eta_2.
  \end{equation}  
  We first prove that $\eta_1(-\tfrac{1}{2}\Delta - \epsilon_0^{(\mathrm{H})}+\ii \omega)^{-1}\eta_2$ is a bounded operator from $L^1(\RR^3)$ to $L^\infty(\RR^3)$ with operator norm $e^{-\max(|\epsilon_0|,|\omega|)R}$. 
  This bound is obtained by using the kernel of the Helmholtz operator $-\tfrac{1}{2}\Delta - \epsilon_0 + \ii \omega$. 
  Since the distance of the supports of $\eta_1$ and $\eta_2$ is at least $R$ and the kernel decays exponentially at a rate $\max(|\epsilon^{(\mathrm{H})}_0|,|\omega|)$, hence 
  \[
    \|\eta_1(-\tfrac{1}{2}\Delta - \epsilon_0^{(\mathrm{H})}+\ii \omega)^{-1}\eta_2\|_{\mathcal{B}(L^1,L^\infty)} \leq C e^{-\max(|\epsilon_0^{(\mathrm{H})}|,|\omega|)R}.
  \]
  By definition of $\Pi$ we have that $\Pi(h^{(\mathrm{H})} - \epsilon_0^{(\mathrm{H})}+\ii \omega)^{-1}$ is bounded in $L^2$ with an operator norm bounded by $\frac{1}{g+|\omega|}$. 
  Since $v \in L^2+L^\infty$ and has a Coulomb-type decay, we obtain the following bound
  \[
    \|\eta_1 \Pi(h^{(\mathrm{H})} - \epsilon_0^{(\mathrm{H})}+\ii \omega)^{-1}\eta_2 \|_{\mathcal{B}(L^1,L^2)} \lesssim e^{-\max(|\epsilon_0^{(\mathrm{H})}|,|\omega|)R} + \frac{1}{R(g+|\omega|)}.
  \]
\end{proof}

\begin{lemma}
  \label{lem:limit_reduced_KH}
  Let $K_R$ be the operator defined in Equation~\eqref{eq:K_R}. Then 
  \begin{equation}
  \lim\limits_{|R|\to \infty}  \int_\RR \tr\Big( \log\big( \id + K_R(\omega) \big) - K_R(\omega) \Big) \, \mathrm{d}\omega= \int_\RR \tr\Big( \log\big( \id - \widetilde{\chi}^{(\mathrm{H})}(\omega) \big) + \widetilde{\chi}^{(\mathrm{H})}(\omega) \Big)\, \mathrm{d}\omega,
  \end{equation}
  where 
  \begin{equation}
    \widetilde{\chi}^{(\mathrm{H})}(\omega) = -2 (v_c^{1/2})^*\phi_0 \Pi \frac{h^{(\mathrm{H})}-\epsilon_0^{(\mathrm{H})}}{(h^{(\mathrm{H})}-\epsilon_0^{(\mathrm{H})})^2 + \omega^2} \Pi \phi_0 v_c^{1/2}.
  \end{equation}
\end{lemma}

\begin{proof}
  We first write 
  \begin{equation}
    2\frac{h^{(\mathrm{H}_2)}-\epsilon^{(\mathrm{H}_2)}_0}{(h^{(\mathrm{H}_2)}-\epsilon^{(\mathrm{H}_2)})^2+\omega^2} = (h^{(\mathrm{H}_2)}-\epsilon^{(\mathrm{H}_2)}_0 + \ii \omega)^{-1} + (h^{(\mathrm{H}_2)}-\epsilon^{(\mathrm{H}_2)}_0 - \ii \omega)^{-1}.
  \end{equation}
  The proof then relies on the second resolvent identity:
  \begin{multline}
    (h^{(\mathrm{H}_2)}-\epsilon^{(\mathrm{H}_2)}_0 - \ii \omega)^{-1} = (-\tfrac{1}{2}\Delta + v^{R} - \epsilon^{(\mathrm{H})}_0-\ii \omega)^{-1} \\
    +  (-\tfrac{1}{2}\Delta + v^{R} - \epsilon^{(\mathrm{H})}_0-\ii \omega)^{-1}(v^{-R}+\epsilon_0^{(\mathrm{H})}-\epsilon_0^{(\mathrm{H}_2)})(h^{(\mathrm{H}_2)} - \epsilon^{(\mathrm{H}_2)}_0-\ii \omega)^{-1}.
  \end{multline}
  It is sufficient to show that 
  \[
    \phi_0^R \Pi_R (-\tfrac{1}{2}\Delta + v^{R} - \epsilon^{(\mathrm{H})}_0-\ii \omega)^{-1}(v^{-R}+\epsilon_0^{(\mathrm{H})}-\epsilon_0^{(\mathrm{H}_2)})(h^{(\mathrm{H}_2)} - \epsilon^{(\mathrm{H}_2)}_0-\ii \omega)^{-1} P \Pi_R \phi_0^R   
  \]
  goes to 0 as $|R|$ goes to $\infty$.

  By assumption on $v$, we can write $v^{-R} = v_2^{-R} + v_\infty^{-R}$, with $\supp v_2^{-R} \subset B_{\tfrac{|R|}{2}}(-R)$ and $\|v_\infty^{-R}\|_{L^\infty} \leq \frac{C}{|R|}$ for some constant $C$ independent of $R$.
  By Proposition~\ref{prop:one-e_model_properties}, we have 
  \begin{equation}
    \big\|\Pi_R (-\tfrac{1}{2}\Delta + v^{R} - \epsilon^{(\mathrm{H})}_0-\ii \omega)^{-1}(v_\infty^{-R}+\epsilon_0^{(\mathrm{H})}-\epsilon_0^{(\mathrm{H}_2)})(h^{(\mathrm{H}_2)} - \epsilon^{(\mathrm{H}_2)}_0-\ii \omega)^{-1} P \Pi_R\big\|_{\mathcal{B}(L^2,L^2)} \leq \frac{C}{|R|(1+\omega^2)}. 
  \end{equation}
  It remains to bound 
  \[
    \phi_0^R \Pi_R (-\tfrac{1}{2}\Delta + v^{R} - \epsilon^{(\mathrm{H})}_0-\ii \omega)^{-1}v_2^{-R}(h^{(\mathrm{H}_2)} - \epsilon^{(\mathrm{H}_2)}_0-\ii \omega)^{-1} P \Pi_R \phi_0^R 
  \]
  Let $\eta_R$ be a smooth cut-off function such that $0 \leq \eta_R \leq 1$, $\supp \eta_R \subset B_{|R|}(R)$ and $\eta_R = 1$ on $B_{\tfrac{|R|}{2}}(R)$. 
  Since $\phi_0^R$ has exponential decay, then there are positive constants $c,C$ independent of $R$ and $\omega$ such that
  \begin{equation}
    \|(v_c^{1/2})^*(1-\eta_R)\phi_0^R \Pi_R (-\tfrac{1}{2}\Delta + v^{R} - \epsilon^{(\mathrm{H})}_0-\ii \omega)^{-1}v_c^{-R}(h^{(\mathrm{H}_2)} - \epsilon^{(\mathrm{H}_2)}_0-\ii \omega)^{-1} P \Pi_R \phi_0^R v_c^{1/2}\|_\mathrm{HS} \leq \frac{Ce^{-cR}}{1+\omega^2}.
  \end{equation}
  Since $\eta_R$ and $v_2^{-R}$ have disjoint support with distance at least $\tfrac{|R|}{2}$, denoting $\bm{1}_{v_2^{-R}}$ the characteristic function of the support of $v_2^{-R}$, we have 
  \begin{align}  
    \|\eta_R \Pi_R(-\tfrac{1}{2}\Delta + v^{R}& - \epsilon^{(\mathrm{H})}_0-\ii \omega)^{-1}v_2^{-R}(h^{(\mathrm{H}_2)} - \epsilon^{(\mathrm{H}_2)}_0-\ii \omega)^{-1} P\|_{\mathcal{B}(L^2,L^2)} \\
    & \leq \|\eta_R \Pi_R((-\tfrac{1}{2}\Delta + v^{R} - \epsilon^{(\mathrm{H})}_0-\ii \omega)^{-1}\bm{1}_{v_2^{-R}}\|_{\mathcal{B}(L^1,L^2)} \|v_2\|_{L^2} \nonumber \\
    & \qquad \qquad \qquad \|(h^{(\mathrm{H}_2)} - \epsilon^{(\mathrm{H}_2)}_0-\ii \omega)^{-1} P\|_{\mathcal{B}(L^2,L^2)} \\
    & \leq \frac{C}{|R|(1+\omega^2)},
  \end{align}
  for some constant $C$ independent of $R$ and $\omega$. 
  Using Lemma~\ref{lem:log_lemma} and~\ref{lem:trace_of_log_vs_hilbert-schmidt}, we conclude that 
  \begin{equation}
    \int_\RR \tr\Big( \log\big( \id + K_R(\omega) \big) - K_R(\omega) \Big) \, \mathrm{d}\omega= \int_\RR \tr\Big( \log\big( \id - \widetilde{\chi_H}(\omega) \big) + \widetilde{\chi}_H(\omega) \Big)\, \mathrm{d}\omega + \cO(\tfrac{1}{|R|^2}).
  \end{equation}
\end{proof}

We have now all the elements to prove Theorem~\ref{thm:phrpa_dissociation}.

\begin{proof}[Proof of Theorem~\ref{thm:phrpa_dissociation}]
  By Proposition~\ref{prop:splitting_phra_Ec}, the limit of the phRPA correlation energy is the sum of the limits of the three terms in Equation~\eqref{eq:splitting_trace}. 
  By Proposition~\ref{prop:limits_rank_one_terms}, the limit of the first two terms is $-\frac{1}{2} \int_{\RR^3 \times \RR^3} |\phi_0(r)|^2 |\phi_0(r')|^2 w(r-r') \, \mathrm{d}r \mathrm{d}r'$.
  Combining Proposition~\ref{lem:first_piece_third_term} with Proposition~\ref{lem:limit_reduced_KH}, we show that in the limit, the last term in Equation~\eqref{eq:splitting_trace} is $2E_c^\mathrm{phRPA}(H)$.
\end{proof}

\section{Conclusion}

In this work, we have put the formula for the phRPA correlation energy on mathematically rigorous footing and proven that it does indeed correctly dissociate H$_2$.  Our formula \eqref{eq:phrpa_correlation_energy} for the phRPA correlation energy properly generalizes the formula typically found in the chemistry literature, which is only valid for Hamiltonians with purely discrete spectrum.  The main change is our formula \eqref{eq:chi0} for $\chi_0$, which generalizes the typical chemistry formula \eqref{eq:chem_chi_0} to the case of a Hamiltonian with continuous spectrum.  In addition to generalizing the formulas to the physically relevant case with continuous spectrum, we also proved that these new formulas are mathematically well-defined.  Once these quantities were rigorously defined, we proved the that the energy of the H$_2$ molecule correctly dissociates in the phRPA approximation when using RHF orbitals to the energy of a single H atom.  Therefore, the phRPA correlation energy preserves a property of the exact XC functional that is \emph{not} preserved by most density functional approximations \cite{cohen2012challenges}.  Finally, we note that our results generalize to dimers of other atoms.

We hope to build on the results of this paper to rigorously prove that the $\cO(R^{-6})$ decay of the van der Waals force is correctly captured by the phRPA correlation energy.  Such a result is known to physicists \cite{fuchs2005describing}, but has not been rigorously proven.  Additionally, we hope to prove results similar to those in this paper for the particle-particle RPA, which is also known in the chemistry literature to correctly dissociate H$_2$ \cite{van2013exchange}.

\bibliography{references}
\bibliographystyle{alpha}

\appendix
\section{Notation}

\begin{table}[h]
  \centering
  \begin{tabular}{ll}
    $x = (r,s) \in \RR^3 \times \ZZ_2$ & Spinorbitals coordinates \\
    $H_N(v_\mathrm{ext},w)$ & Interacting $N$-body Hamiltonian (Equation~\eqref{eq:N-body-Hamiltonian}) \\
    $(E_0,\Psi_0)$ & Lowest eigenpair of $H_N(v_\mathrm{ext},w)$ \\
    $H_N(v_0)$ & Noninteracting $N$ body Hamiltonian ($H_N(v_0) = H_N(v_0,0)$) \\
    $(\epsilon_k^{(\mathrm{H})},\phi_k)$ & Eigenpairs of $h^{(\mathrm{H})}$ (Equation~\eqref{eq:one_electron_model_H}) \\ 
    $(\epsilon_k^{(\mathrm{H}_2)},\psi_k)$ & Eigenpairs of $h^{(\mathrm{H}_2)}$ (Equation~\eqref{eq:one_electron_model_H2}) \\ 
    $\chi$ & Interacting retarded linear response function (Equation~\eqref{eq:chi_rigorous}) \\
    $\chi_0$ & Interacting retarded linear response function (Equation~\eqref{eq:chi0}) \\
    $\widetilde{A}(z)$ & Laplace transform of $A$ (Equation~\eqref{eq:laplace_transform}) \\
    $\phi_k^R(r) = \phi_k(r-R)$ & Translation by $R$
  \end{tabular}
  \caption{Table of notation}
\end{table}

\section{Retarded linear response function $\chi$}

\begin{proposition}
  \label{prop:rigorous_chi_stronger_assumptions}
  Let $\Psi$ be the solution of the time-dependent Schr\"odinger equation~\eqref{eq:TD_schrodinger} and $\alpha, \beta \in C_c^\infty(\RR^3)$ and real-valued. Then 
  \begin{equation}
    \Big\langle \Psi(t), \sum_{i=1}^N \alpha(r_i) \Psi(t) \Big\rangle = \Big\langle \Psi_0, \sum_{i=1}^N \alpha(r_i) \Psi_0 \Big\rangle + \varepsilon (f \star \langle \alpha, \chi \beta \rangle)(t) + R_2(\varepsilon),
  \end{equation}
  where $\chi$ is the operator defined in Equation~\eqref{eq:chi_rigorous} and $|R_2(\varepsilon)| \leq C \varepsilon^2 |t| \|f\|_{L^\infty} \|\alpha\|_{L^\infty} \|\beta\|_{L^\infty}$ for some constant $C$ only depending on $N$.
\end{proposition}

\begin{proof}
We denote $H_N = H_N(v_\mathrm{ext},w)$.
By Duhamel's principle, the solution to the time-dependent Schrödinger equation~\eqref{eq:TD_schrodinger} is 
\begin{equation}
  \label{eq:solution_time_dependent_schrodinger}
  \Psi(t) = e^{-\ii H_Nt} \Psi_0 - \ii \varepsilon \int_0^t e^{-\ii (t-s)H_N} f(s) \mathfrak{B} \Psi(s) \, \mathrm{d}s,
\end{equation}
  where $\mathfrak{B} = \sum_{i=1}^N \beta(r_i)$. 
  Since $H_N$ is self-adjoint, we have 
  \[
    \Big| \int_0^t e^{-\ii (t-s)H_N} f(s) \mathfrak{B} \Psi(s) \, \mathrm{d}s \Big| \leq |t| N \|f\|_{L^\infty} \|\beta\|_{L^\infty}.  
  \]
  Denoting by $\mathfrak{a} = \sum_{i=1}^N \alpha(r_i)$, we have 
\begin{align*}
  \Big\langle \Psi(t),  \mathfrak{a} \Psi(t) \Big\rangle &= \Big\langle \Psi_0, \mathfrak{a} \Psi_0 \Big\rangle - \ii \varepsilon \Big\langle \Psi_0, \mathfrak{a}(t) \int_0^t f(s) e^{\ii s H_N} \mathfrak{B} \Psi(s) \, \mathrm{d}s \Big\rangle \\
  &\qquad \qquad
  + \ii \varepsilon \, \Big\langle \int_0^t f(s) e^{\ii s H_N} \mathfrak{B} \Psi(s) \, \mathrm{d}s , \mathfrak{a}(t)  \Psi_0 \Big\rangle + R_2(\varepsilon),
\end{align*}
where $\mathfrak{a}(t) = \exp(\ii H_Nt) \mathfrak{a} \exp(-\ii H_Nt)$ and $R_2(\varepsilon)$ as in the proposition.
Inserting the Duhamel formula again in the equation above, we obtain 
\begin{equation}
 \Big\langle \Psi(t),  \mathfrak{a} \Psi(t) \Big\rangle =  \Big\langle \Psi_0, \mathfrak{a} \Psi_0 \Big\rangle -\ii \varepsilon \Big\langle \Psi_0, \int_0^t f(s) \big[\mathfrak{a}(t), \mathfrak{B}(s)\big] \, \mathrm{d}s \ \Psi_0 \Big\rangle + R_2(\varepsilon),
\end{equation}
with $\mathfrak{B}(s) = \exp(\ii H_Ns) \mathfrak{a} \exp(-\ii H_N s)$.
Using that $\Psi_0$ is the ground-state of $H_N$, and setting $P_0 = \id - |\Psi_0\rangle \langle \Psi_0|$, the expression of the commutator can be simplified
\begin{multline}
  \Big\langle \Psi_0 , \big[ \mathfrak{a}(t), \mathfrak{B}(s) \big] \Psi_0 \Big\rangle =  \Big\langle \Psi_0 , \mathfrak{a}P_0\exp(-\ii (H_N-E_0)(t-s)) P_0 \mathfrak{B}\Psi_0 \Big\rangle \\
   -\Big\langle \Psi_0 , \mathfrak{B}P_0\exp(\ii (H_N-E_0)(t-s)) P_0\mathfrak{a} \Psi_0 \Big\rangle .
\end{multline}
Inserting this in the expression of $\Big\langle \Psi(t), \mathfrak{a} \Psi(t) \Big\rangle$, we get
\begin{align}
  \Big\langle \Psi(t), \mathfrak{a} \Psi(t) \Big\rangle &= \big\langle \Psi_0, \mathfrak{a} \Psi_0 \big\rangle 
  - \ii \varepsilon \int_0^t f(s) \Big(  \Big\langle \Psi_0 , \mathfrak{a}P_0\exp(-\ii (H_N-E_0)(t-s)) P_0 \mathfrak{B}\Psi_0 \Big\rangle \\
  & \qquad \qquad - \Big\langle \Psi_0 , \mathfrak{B}P_0\exp(\ii (H_N-E_0)(t-s)) P_0\mathfrak{a} \Psi_0 \Big\rangle \Big) \mathrm{d}s + R_2(\varepsilon) \\
  & = \big\langle \Psi_0, \mathfrak{a} \Psi_0 \big\rangle + \varepsilon (f \star \langle \alpha, \chi(t) \beta \rangle) + R_2(\varepsilon),
\end{align}
where $\star$ denotes the convolution on $\RR$ and $\chi(\tau)$ is the operator given by 
  \begin{align}
    \langle \alpha, \chi(\tau) \beta \rangle &= 2 \, \mathrm{Re}\Big( -\ii \theta(\tau) \Big\langle \Psi_0 , \mathfrak{a}P_0\exp(-\ii (H_N-E_0)\tau) P_0 \mathfrak{B}\Psi_0 \Big\rangle\Big),
  \end{align}
\end{proof}

\end{document}